\documentclass[a4paper,reqno,11pt,oneside]{amsart}
\usepackage{amsmath,geometry}
\usepackage{graphicx}
\usepackage{mathrsfs}
\usepackage{amssymb,amsmath, amsfonts, amsthm}
\usepackage{latexsym}
\usepackage{color} %{\color{red}..............OK}
\usepackage[dvips]{epsfig}
\usepackage[colorlinks]{hyperref}
\usepackage{pgfplots}
\usepackage{bm}
\usepackage{a4wide}
\usepackage[numbers,sort&compress]{natbib}

\usepackage{color}

\numberwithin{equation}{section}
\usepackage{mathrsfs}
\usepackage{amsfonts}
\usepackage{amsmath}
\usepackage{stmaryrd}
\usepackage{amssymb}
\usepackage{amsthm}
\usepackage{mathrsfs}
\usepackage{url}
\usepackage{amsfonts}
\usepackage{amscd}
\usepackage{indentfirst}
\usepackage{enumerate}
\usepackage{amsmath,amsfonts,amssymb,amsthm}
\usepackage{amsmath,amssymb,amsthm,amscd}
\usepackage{graphicx,mathrsfs}
\usepackage{appendix}
\usepackage[numbers,sort&compress]{natbib}
\usepackage{color}
\usepackage{verbatim}
 \usepackage{titlesec}

 \titleformat{\section}
   {\centering\Large\bfseries} % 锟斤拷锟斤拷锟酵达拷小
  {\thesection}{1em}{}

\allowdisplaybreaks[4]

\newcommand{\R}{\mathbb{R}}

\renewcommand{\theequation}{\arabic{section}.\arabic{equation}}
\renewcommand\thesection{\arabic{section}}
\makeatletter
\renewcommand\subsection{\@startsection{subsection}{2}{\z@}%
  {-3.25ex\@plus -1ex \@minus -.2ex}%
  {1.5ex \@plus .2ex}%
  {\normalfont\large\bfseries\itshape}}
\makeatother
\newtheorem{Thm}{Theorem}[section]
\newtheorem{Lem}[Thm]{Lemma}

\newtheorem{Prop}[Thm]{Proposition}

\theoremstyle{definition}
\newtheorem{Def}[Thm]{Definition}
\newtheorem{Rem}[Thm]{Remark}

\begin{document}

\title[Segregated solutions for a class of systems with  Lotka-Volterra interaction]{Segregated solutions for a class of systems with  Lotka-Volterra interaction}

\author{Qing Guo}
\address[Qing Guo]{College of Science, Minzu University of China, Beijing 100081, China}
\email{guoqing0117@163.com}

\author{Angela Pistoia}
\address[Angela Pistoia]{Dipartimento SBAI, Sapienza Universit\`a di Roma,
via Antonio Scarpa 16, 00161 Roma, Italy. }
\email{angela.pistoia@uniroma1.it}

  \author{Shixin Wen}
  \address[Shixin Wen]{School of Mathematics and Statistics, Central China Normal University, Wuhan 430079, China} \email{sxwen@mails.ccnu.edu.cn}

\thanks{A.Pistoia has been   partially supported by
the MUR-PRIN-20227HX33Z ``Pattern formation in nonlinear phenomena'' and
the INdAM-GNAMPA group. Q. Guo has been supported by the National Natural Science Foundation of China (Grant No. 12271539), and this research was also supported by the China Scholarship Council (CSC) during the visits of the first and third authors to Sapienza University of Rome.}



\begin{abstract}
This paper deals with the existence of positive solutions to the system
    \[
    \begin{cases}
    -\Delta w_1 - \varepsilon w_1 = \mu_{1} w_1^{p} + \beta w_1 w_2,\quad w_1> 0, & \text{in } \Omega, \\
    -\Delta w_2 - \varepsilon w_2 = \mu_{2} w_2^{p} + \beta w_1 w_2,\quad w_2> 0, & \text{in } \Omega, \\
    w_1 = w_2 = 0, & \text{on } \partial \Omega,
    \end{cases}
    \]
    where \( \Omega \subseteq \mathbb{R}^{N} \), \( N \geq 4\), \( p = 2^*-1 \), and \( \varepsilon \in (0, \Lambda_1(\Omega)) \) is sufficiently small. The interaction coefficient \( \beta = \beta(\varepsilon) \to 0 \) as \( \varepsilon \to 0 \).
    
    We construct a family of \emph{segregated solutions} to this system, where each component blows-up at a different critical point of the Robin function as \( \varepsilon \to 0 \). The system lacks a variational formulation due to its specific coupling form,
     which leads to essentially different behaviors in the subcritical, critical, and supercritical regimes and requires an 
  appropriate functional settings to carry out the construction.

\end{abstract}

\date{}\maketitle

\section{Introduction}

This paper is concerned with the construction of positive solutions to a class of Lotka--Volterra-type elliptic systems with critical nonlinearities and weak interspecies interaction, motivated by models in population dynamics where only nonnegative population densities are biologically meaningful. Specifically, we consider the following boundary value problem:
\begin{equation}\label{lv}
    \begin{cases}
    -\Delta w_1 - \varepsilon w_1 = \mu_{1} w_1^{p} + \beta w_1 w_2,\quad w_1> 0, & \text{in } \Omega, \\
    -\Delta w_2 - \varepsilon w_2 = \mu_{2} w_2^{p} + \beta w_1 w_2,\quad w_2> 0, & \text{in } \Omega, \\
    w_1 = w_2 = 0, & \text{on } \partial \Omega, \\
    \end{cases}
    \end{equation}
    where $\Omega \subseteq \mathbb{R}^{N}$, \( N \geq 4 \), \( p = 2^* - 1 \), \( \varepsilon \in (0, \Lambda_1(\Omega)) \) is sufficiently small, and the interaction coefficient \( \beta := \beta (\varepsilon) \to 0 \) as \( \varepsilon \to 0 \).

System \eqref{lv} arises as a spatial extension of the classical Lotka--Volterra model, which was originally introduced to describe the population dynamics of interacting species. In this setting, \( w_1(x) \) and \( w_2(x) \) represent the spatial densities of two species at location \( x \in \Omega \). The Laplacian terms \( -\Delta w_i \) model random movement or diffusion of individuals, accounting for their tendency to disperse from overcrowded regions. The linear terms \( -\varepsilon w_i \) represent constant per capita mortality or environmental suppression. The nonlinear terms \( \mu_i w_i^p \) represent intra-specific growth regulation: when \( \mu_i > 0 \), they model cooperative effects that enhance growth at higher densities, while negative values \( \mu_i < 0 \) would correspond to inhibitory effects arising from overpopulation or resource limitation.
The coupling term \( \beta w_1 w_2 \) models interspecies interaction: when \( \beta > 0 \), the interaction is mutualistic (beneficial to both species), while \( \beta < 0 \) corresponds to competition, where the presence of one species inhibits the growth of the other.
 The Dirichlet boundary conditions \( w_i = 0 \) on \( \partial \Omega \) correspond to an inhospitable habitat boundary, outside of which the species cannot survive or establish.

Lotka--Volterra-type reaction--diffusion systems have been widely used in ecological modeling to study the formation of spatial patterns, coexistence of species, and the impact of movement and interactions. In the context of weak cooperation, several existence and uniqueness results have been obtained in the subcritical regime using variational or topological methods; see, for example, \cite{CantrellCosner1987, KormanLeung1987, AliCosner1992, Korman1992}. More general models incorporating spatial heterogeneity and transport phenomena have been studied in \cite{Delgado2000}. For non-variational structures, López-Gómez and collaborators \cite{LopezGomez1992, LopezGomez1994} developed a global bifurcation framework to analyze the existence of positive solutions for generalized Lotka--Volterra systems.

From a mathematical perspective, system \eqref{lv} presents substantial challenges. Although the coupling term \( \beta w_1 w_2 \) is symmetric in \( w_1 \) and \( w_2 \), it does not correspond to consistent variational derivatives in both equations. As a result, the system cannot be derived as the Euler--Lagrange equations of any energy functional and thus lacks a variational structure. This precludes the direct application of critical point theory or energy minimization techniques, requiring alternative tools such as bifurcation methods, fixed-point arguments, or singular perturbation techniques.

Another difficulty in our problem originates from the critical growth of the self-interaction terms \( \mu_i w_i^p \), where \( p = 2^*-1 \). This exponent corresponds to the critical Sobolev embedding and leads to a loss of compactness, even in the scalar case. A classical example is the celebrated Brezis--Nirenberg problem~\cite{Brezis83}, which has inspired a large body of works in the literature.
In particular, notable contributions on  the existence and asymptotic properties of solutions as \( \varepsilon \to 0 \) 
has been obtained by 
~\cite{Han91,Rey89,Rey90,musso,kl1,kl2,druet,musa,pirava,piro}.
Our work can be viewed as a generalization of this theory to coupled systems.

In the context of systems with variational structure—such as coupled nonlinear Schrödinger equations—there are now several results addressing critical and near-critical coupling; see, for instance, \cite{ChenLin14,ChenZou12,Pistoia2017}. However, the Lotka--Volterra system under consideration presents additional challenges. The coupling term \( \beta w_1 w_2 \), though symmetric in form, does not arise from an energy functional and exhibits dimension-dependent behavior: it is subcritical for \( N =3, 4, 5 \), critical for \( N = 6 \), and supercritical for \( N \geq 7 \). We point out that the  case \( N = 3 \)  differs entirely from case \(N\geq 4\), as already observed in the case of the single equation.
That is why in the present paper we only study the case  \(N\geq 4\).
Moreover, we focus on the small coupling regime, where \( \beta := \beta(\varepsilon) \to 0 \) as \( \varepsilon \to 0 \).

\vspace{0.2cm}
Before stating our result, we need to introduce some notation.
Let us denote by \( G \) the Green's function of the negative Laplacian on \( \Omega \) with homogeneous Dirichlet boundary conditions, and by \( H \) its regular part, defined such that
\begin{align*}
 H(x, y) = \frac{B_N}{|x - y|^{N-2}} - G(x, y), \quad \forall (x, y) \in \Omega^2,   
\end{align*}
where \( B_N = \left[(N-2)\, \text{meas}(S^{N-1})\right]^{-1} \), and \( S^{N-1} \) denotes the \((N-1)\)-dimensional unit sphere in \( \mathbb{R}^N \). 
For every \( x \in \Omega \), the leading term of \( H \), defined by
\begin{align*}
  \Phi_{\Omega}(x) := H(x, x),  
\end{align*}
is called the \emph{Robin function} of \( \Omega \) at the point \( x \).

In the following, we remind the definition of {\em stable critical point}. 
\begin{Def}\label{yy1}
Given a smooth function $f:D\subset\mathbb R^N\to\mathbb R,$ a  critical point $x_0\in D$ is  stable if there exists a neighbourhood $\Theta$ of $x_0$ such that the Brouwer degree  $\mathtt{deg}(\nabla f,\Theta,0)\not=0,$ where 
$\mathtt{deg}$ denotes  the Brouwer degree. \end{Def}
In particular, any  strict local minimum or maximum points and any non-degenerate critical points are a stable critical points.
 \\
 
 Finally, our main results can be stated as follows.

\medskip
\begin{Thm}\label{th1-4}
   Assume that $N\geq4$, and let \( \bar{\xi}_1,  \bar{\xi}_2 \in \Omega \) be two stable critical points of the Robin function \( \Phi_\Omega \). 
   Then there exists \( \varepsilon_0 > 0 \), such that for any \( \varepsilon\in (0, \varepsilon_0) \), there exists $\beta_\varepsilon>0$ with the following property: for every  $\beta\in\left[-\beta_\varepsilon,\beta_\varepsilon\right]$, the system \eqref{equivalent-lv} admits a positive segregated solution \( (u, v) \) which blows-up at \(   \bar{\xi}_1 \) and \(  \bar{\xi}_2 \) as $\varepsilon\to0$, respectively.
    \end{Thm}

    \medskip
Let us make some comments.\\

\begin{Rem}\rm
We observe that  there is a wide class of domains whose Robin function satisfies the assumption of Theorem \ref{th1-4}.
In fact, it is well known that   the Robin function has two critical points in a dumbell shaped domain
 (two local minima) and in a domain with a hole   (one minimum and one saddle point).
Moreover, it is always possible to perturb a bit the domain  so that  all the critical points of the Robin function in the perturbed domain  are nondegenerate as proved by Micheletti and Pistoia \cite{mipi}. 
\end{Rem} 

\begin{Rem}\rm
   We use a finite-dimensional reduction approach to construct our solutions that exhibit the segregated profile as $\varepsilon$ approaches zero.
   As it is well known, the problem is strongly affected by the space dimension $N.$ In fact, when $N=4$, the rate
    of the error term is not small enough to argue as in the higher-dimensional cases,  then some new local Pohozaev identities are introduced to overcome the difficulties.  Moreover, in the supercritical coupling regime (i.e., \( N \geq 7 \)), working within standard energy spaces is no longer sufficient, as the coupling term becomes ill-defined. To overcome this difficulty, we are forced to introduce alternative integrability spaces, which significantly increase the analytical complexity of the problem.
     \end{Rem}

\begin{Rem}\rm
We point out that the coupling term \( \beta w_1w_2 \) is  linear and thus induces only a weak interaction, which strongly affects the 
 leading term of the reduced problem. The smallness condition on \( \beta \) becomes essential in our construction, since it ensures that   the coupling term becomes a higher-order term in the expansion (see the terms $E_8$ and $Q_8$ in the proof of Proposition \ref{supercritical} in  Appendix \ref{BBB} and the terms $J_1$ and $J_2$ in the proof of Lemma \ref{newlemma}). It would be interesting to understand what happens in presence of a strong interaction but we believe that the profile of the solutions we are building must be changed drastically.
   \end{Rem}

\begin{Rem}\rm
  Finally, we observe that our  construction  can be carried out for a system with $k$ components provided the interaction among different components is small enough. In particular, we can claim that  there exists a solution whose components blow-up at   $k$ different critical points of the Robin function as $\varepsilon\to0.$
    \end{Rem}

The paper is organized as follows. In Section \ref{two} we set up the problem, while in Section \ref{three} we provide the proof of the main theorem. All intermediate results required in the finite-dimensional reduction process, which appear in Section \ref{three}, are proved in the appendices \ref{AAA} and \ref{BBB}. Appendix \ref{CCC} contains well‑known results, which we reproduce here to keep the work self‑contained.

%We observe that the Lotka–Volterra type interaction term \( \beta w_1 w_2 \) exhibits different critical behaviors depending on the spatial dimension:\begin{itemize} \item[-] For \( N = 4, 5 \), the coupling term is subcritical; \item[-] For \( N = 6 \), the coupling term is critical; \item[-] For \( N \geq 7 \), the coupling term is supercritical.\end{itemize}

 %Hence, by adopting ideas from \cite{dcds} ($N=4$), \cite{musso} ($N=5,6$) and \cite{ade} ($N\geq 7$), we perform the results respectively. 

\section{Preliminaries}\label{two}
\vspace{0.2cm}\subsection{Notations}
Let $H_0^1(\Omega_{\lambda})$ be the Hilbert space, where 
\begin{align*}
   \Omega_{\lambda}:=\frac{\Omega}{\lambda}, \quad\text{$N\geq 4$}.
\end{align*}
We would like to point out that the parameter $\lambda$ is related to $\varepsilon$ in our problem, that is, $\lambda:=\lambda(\varepsilon)\to 0$ as $\varepsilon\to 0$. The precise relationship between them will be clarified later, depending on the dimension of space.

In addition, the corresponding inner product and the norm are given by
$$
\langle u, v\rangle_{H_0^1( \Omega_{\lambda})}=\displaystyle\int_{ \Omega_{\lambda}} \nabla u \cdot \nabla v,\quad\|u\|_{H_0^1(\Omega_{\lambda})}=\left(\int_{\Omega_{\lambda}}|\nabla u|^2\right)^{\frac{1}{2}}.
$$
For any $s \in[1,+\infty)$, the space $L^s(\Omega_{\lambda})$ is also equipped with the standard norm
$$
\|u\|_{s}=\left(\displaystyle\int_{\Omega_{\lambda}}|u|^s\right)^{\frac{1}{s}}.
$$

\medskip
Moreover, for $N\geq 4$, we denote ${X}:={H}_0^1(\Omega_{\lambda}) \cap {L}^s(\Omega_{\lambda})$, where 
\begin{equation}\label{X-space}
    \begin{cases}
        s=2^*,&\quad \text{if $N=4,5$ (subcritical case) or $N=6$ (critical case),}\\
          s>2^*, &\quad \text{if $N\geq 7$ (supercritical case).}\\
    \end{cases}
\end{equation} 
${X}$ is a Banach space equipped with the norm
$$
\|u\|_{X}=\|u\|_{H_0^1(\Omega_{\lambda})}+\|u\|_{s}.
$$ We  point out that when $N=4,5,6$, the space ${X}$ coincides with ${H}_0^1(\Omega_{\lambda})$. For the sake of simplicity, when $N\geq 7$, we will choose $s=N$.
 Next, we define $H:=X\times X$ equipped with the norm
$$\left\|(u,v)\right\|_{H}=\|u\|_{X}+\|v\|_{X}.$$

 Set $y=x/\lambda\in\Omega_\lambda$, $\forall x\in\Omega\subseteq \R^N$, $N\geq 4$. A
direct computation shows that, $(w_1, w_2)$ solves problem \eqref{lv} if and only if 
\begin{align}\label{u-w}
    \left({u},{v}\right):=\left(\lambda^{\frac{N-2}{2}}\,w_1(\lambda y),\,\lambda^{\frac{N-2}{2}}\,w_2(\lambda y)\right)
\end{align}
solves
\begin{equation}\label{equivalent-lv}
\begin{cases}
-\Delta {u}-\varepsilon\,\lambda^{2} {u}=\mu_{1} {u}^{p}+\beta\,\lambda^{\frac{6-N}{2}} \,{u}\,{v}, \quad {u}>0, &\text{ in $\Omega_\lambda$},
\\
-\Delta {v}-\varepsilon\,\lambda^{2} {v}=\mu_{2} {v}^{p}+\beta\,\lambda^{\frac{6-N}{2}} \,{u}\,{v},\quad {v}>0,  &\text{ in $\Omega_\lambda$}, \\
{u}={v}=0,&\text{ on $\partial\Omega_\lambda$,}
\end{cases}\end{equation}
 where $\Omega_\lambda \subseteq \mathbb{R}^{N}$, \( N \geq 4 \), and \( p = 2^* - 1 \). 

 \medskip
Assume that $\xi_1,\,\xi_2\in\Omega$. In this paper, we aim to construct the solution to \eqref{equivalent-lv} of the form $$(u,v)=\left(\mu_1^{-\frac{N-2}{4}} P_{\lambda} U_{\delta_1, \xi_1/\lambda}+\phi_1 ,\,\mu_2^{-\frac{N-2}{4}} P_{\lambda} U_{\delta_2, \xi_2/\lambda}+\phi_2 \right),$$ 
 \text{with} $\|\phi_1\|_{X},  \|\phi_2\|_{X}\to 0$ \text{as } $\lambda:=\lambda(\varepsilon) \to 0$. Here $P_{\lambda}:D^{1,2}(\mathbb{R}^N)\to H^{1}_0(\Omega_\lambda)$ is the projection operator, and $\delta_i$ ($i=1,2$) are fixed positive constants. Moreover,
$\mu_i^{-\frac{N-2}{4}} P_{\lambda}U_{\delta_i, \xi_i/\lambda} $ solves 
\begin{equation*}
\begin{cases}
-\Delta \left(\mu_i^{-\frac{N-2}{4}} P_{\lambda}U_{\delta_i, \xi_i/\lambda}\right)= \mu_i\left(\mu_i^{-\frac{N-2}{4}} U_{\delta_i, \xi_i/\lambda}\right)^{p}, \quad \text{in $\Omega_\lambda$},\\
\mu_i^{-\frac{N-2}{4}} P_{\lambda}U_{\delta_i, \xi_i/\lambda}=0,\quad \text{on $\partial\Omega_\lambda$,}
\end{cases}
\end{equation*}
 which is closely related to the limit equation
    \begin{equation*}
       -\Delta U_{  {\delta}, {\xi}}= U_{ {\delta},{\xi}}^{p}, \quad \text{in $\mathbb{R}^N$, }
    \end{equation*}
    whose solutions are (see \cite{aubin,talenti})
    \begin{align*}
       U_{ {\delta}, {\xi}}(x)=C_N\,\frac{{ {\delta}}^{\frac{N-2}{2}}}{\left( {\delta}^2+\left|x- {\xi}\right|^2\right)^{\frac{N-2}{2}}}\in D^{1,2}(\mathbb{R}^N).
    \end{align*}

\vspace{0.2cm}

\vspace{0.2cm}
\subsection{Setting of the problem}
\begin{Def}
For $N\geq4$, let $i_\lambda^{*}:L^{\frac{2N}{N+2}}(\Omega_\lambda)\to H^{1}_0(\Omega_\lambda)$ be the adjoint operator of the embedding $i_\lambda:H^{1}_0(\Omega_\lambda)\hookrightarrow L^{\frac{2N}{N-2}}(\Omega_\lambda)$, i.e.
    $$
i_\lambda^*(f)=u \quad \Longleftrightarrow \quad \left\langle u, \varphi\right\rangle_{{H}_0^1\left(\Omega_{\lambda}\right)}=\displaystyle\int_{\Omega_{\lambda}} f(x)\, \varphi(x) d x, \quad \forall \varphi \in {H}_0^1\left(\Omega_{\lambda}\right).
$$
\end{Def}

In the following, we present some embedding properties of the operators, which can be found in \cite[Remark 2.2 and Lemma 2.3]{ade}.
\begin{Lem}\label{embed} The operator $i_\lambda^*$ is uniformly continuous with respect to $\lambda>0$, that is, 
$$\left\|i_\lambda^{*}(u)\right\|_{H_0^1(\Omega_\lambda)}\leq\,S^{-\frac{1}{2}}\left\|u\right\|_{{\frac{2N}{N+2}}},\quad \forall u\in L^{\frac{2N}{N+2}}(\Omega_\lambda),$$
where $S$ is the best constant for the Sobolev embedding. 
\end{Lem}
\begin{Lem}
\label{ls}
Let $s>\frac{N}{N-2}$. If $u \in {L}^{\frac{Ns}{N+2s}}\left(\Omega_{\lambda}\right) \cap \mathrm{L}^{\frac{2 N}{N+2}}\left(\Omega_{\lambda}\right)$, then $i_\lambda^*(u) \in$ $L^s\left(\Omega_{\lambda}\right)$. Moreover,
$$
\left\|i_\lambda^*(u)\right\|_{s}\leq C(\Omega)\|u\|_{{\frac{N s}{N+2 s}}},
$$
where $C(\Omega)>0$ is a constant depending only on the domain $\Omega$. 
\end{Lem}
\begin{Def}
    Given any $\eta\in(0,1)$. Define the following set
$$\mathcal{O}_\eta= \left\{\left(\bm{\delta}, \bm{\xi}\right) \in\left(\mathbb{R}^{+}\right)^2 \times \Omega^2:\,\left|\xi_{1}-\xi_{2}\right|\geq 2\eta,\,dist(\xi_i,\partial\Omega)\geq 2\eta,\,\eta<\delta_i<1 / \eta,\quad i=1,2\right\} .$$
\end{Def}
\vspace{0.2cm}
Now, by the definition of the operator $i_\lambda^*$, problem \eqref{equivalent-lv} reduces to
\begin{equation}\label{main-geq5}
    \begin{cases}
      u=i_\lambda^{*}\left(\mu_{1} {g(u)}+\beta\lambda^{\frac{6-N}{2}} uv+\varepsilon\,\lambda^{2} u\right),\quad u>0,\quad&\text{ in $\Omega_\lambda$},
\\
v= i_\lambda^{*}\left(\mu_{2} {g(v)}+\beta\lambda^{\frac{6-N}{2}} uv+\varepsilon\,\lambda^{2} v\right),\quad v>0, &\text{ in $\Omega_\lambda$} , \\
u=v=0,&\text{ on $\partial\Omega_\lambda$,}\\ 
(u,v)\in H.
    \end{cases}
\end{equation}
where $\Omega\subseteq\R^N$, $N\geq 4$, and $g(\omega)=\left(\omega^{+}\right)^p$, \( p = 2^* - 1 \).

\vspace{0.2cm}
    Before studying \eqref{main-geq5}, we fix some notations for simplicity. For any $\left(\bm{\delta}, \bm{\xi}\right)\in \mathcal{O}_\eta$, set $\bm{y}:=\left(y_1,y_2\right) \in \left(\Omega_{\lambda}\right)^{2}$, where $y_i=\xi_{i} / \lambda$, $i=1,2$. Define
\begin{align}
    \label{limit-def}
U_i:=U_{\delta_i, y_i}, \quad \text { and } \quad P_{\lambda}U_i:=i_\lambda^*\left(U_{\delta_i, y_i}^p\right).
\end{align}
Moreover, 
for $j=0,1, \ldots, N$ and $i=1,2$, we denote
\begin{align}  \label{ker-ele}
    \psi_i^0:=\frac{\partial U_{\delta_i, y_i}}{\partial \delta_i}, \quad \psi_i^j:=\frac{\partial U_{\delta_i, y_i}}{\partial y_i^j}, \quad \text { and } \quad P_{\lambda} \psi_i^j:=i_\lambda^*\left(p\,U_{\delta_i, y_i}^{p-1} \psi_i^j\right).
\end{align}

Now, to proceed with the reduction, we split the space $X$ into the sum of spaces $K_i$ and $K_i^{{\perp}}$, where 
\begin{equation}\label{space}
    \begin{split}
    K_i&:=K^{\lambda}_{\delta_i,y_i}=span\left\{P_\lambda\psi^{j}_{\delta_i,y_i}(x),\quad j=0,1,\cdots,N\right\},\\
    K_i^{{\perp}}&:=\left\{\phi\in X\Big|\,\left\langle \phi, P_\lambda\psi^{j}_{\delta_i,y_i}(x)\right\rangle_{X}=0,\quad j=0,1,\cdots,N\right\}.
    \end{split}
\end{equation}
Furthermore, 
we introduce the projection maps $${\Pi_{\bm{\delta},\bm{\xi}}}:={\Pi_1}\times {\Pi_2}:H \mapsto  {K_1}\times K_2,$$
and
$$\Pi^{\perp}_{\bm{\delta},\bm{\xi}}:={\Pi^{\perp}_1}\times {\Pi^{\perp}_2}:H\mapsto \left({K_1}\times K_2\right)^{\perp}={K^{\perp}_1}\times K^{\perp}_2.$$ In addition, for $i=1,2$, $\Pi_{i}:X \mapsto  {K_i}$ with
$$
\Pi_{i}\left(w\right)=\sum_{j=0}^N\left\langle w, P_{\lambda} \psi_{\delta_i,y_i}^j\right\rangle_{X}P_{\lambda} \psi_{\delta_i, y_i}^j,\quad \text{$\forall\,w\in X$,}
$$
and $\Pi_{i}^{\perp}:X\to K_{i}^{\perp}$, that is, 
$$\Pi_{i}^{\perp}(w)=w-\Pi_{i}(w),\quad \forall \,w\in X.$$
Obviously, it holds that
$$
\left\|\Pi^{\perp}_{\bm{\delta},\bm{\xi}}\left(u,v\right)\right\|_{H}\lesssim\left\|(u,v)\right\|_H, \quad \forall (u,v) \in H.
$$

Our approach to solve the problem \eqref{main-geq5} will be to find, for suitable $\bm{\delta}>0,\,\bm{\xi} \in \Omega^2$ and for small $\lambda(\varepsilon)>0$, there exists $\bm{\phi}:=\left(\phi_1,\phi_{2}\right) \in {K^{\perp}_1}\times K^{\perp}_2$ such that 
\begin{align}\label{desired}
    (u,v):=\left(\mu_1^{-\frac{N-2}{4}}P_\lambda U_{\delta_1,\xi_1/\lambda}+\phi_1,\mu_2^{-\frac{N-2}{4}}P_\lambda U_{\delta_2,\xi_2/\lambda}+\phi_2\right)
\end{align}verifying 
\begin{equation}\label{first-equ}
\begin{aligned}
  \begin{cases}
   \Pi^{\perp}_{1}\left(u-i_\lambda^*\left[\mu_{1} {g(u)}+\beta\lambda^{\frac{6-N}{2}} uv+\varepsilon\,\lambda^{2} u\right]\right)=0,  \\
\Pi^{\perp}_{2}\left(v-i_\lambda^*\left[\mu_{2} {g(v)}+\beta\lambda^{\frac{6-N}{2}} uv+\varepsilon\,\lambda^{2}v\right]\right)=0,
\end{cases}\\  
\end{aligned}
\end{equation}
and
\begin{equation}\label{second-equ}
  \begin{aligned}
\begin{cases}
 \Pi_{1}\left(u-i_\lambda^*\left[\mu_{1} {g(u)}+\beta\lambda^{\frac{6-N}{2}} uv+\varepsilon\,\lambda^{2} u\right]\right)=0,  \\
\Pi_{2}\left(v-i_\lambda^*\left[\mu_{2} {g(v)}+\beta\lambda^{\frac{6-N}{2}} uv+\varepsilon\,\lambda^{2}v\right]\right)=0.   
\end{cases}
\end{aligned}  
\end{equation}

 \vspace{0.2cm}

 Firstly, we study the solvability of \eqref{first-equ}. Plugging \eqref{desired} into \eqref{first-equ}, we rewrite \eqref{first-equ} as
\begin{equation}\label{new-linear}
\bm{\mathcal{L}}_{\bm{\delta},\bm{\xi}}\left(\bm{\phi}\right)=\bm{\mathcal{E}}_{\bm{\delta},\bm{\xi}}+\bm{\mathcal{N}}_{\bm{\delta},\bm{\xi}}\left(\bm{\phi}\right).
\end{equation}
Here, the linear operator $\bm{\mathcal{L}}_{\bm{\delta},\bm{\xi}}=\left(\mathcal{L}^{1}_{\bm{\delta},\bm{\xi}},\mathcal{L}^{2}_{\bm{\delta},\bm{\xi}} \right)$ is given by 
$$
\mathcal{L}^{i}_{\bm{\delta},\bm{\xi}}:=\Pi^{\perp}_i\mathcal{L}_i:{K}^{\perp}_{1}\times{K}^{\perp}_{2}\to{K}^{\perp}_{i},\quad i=1,2,
$$
 where
 $\mathcal{L}_i$ is defined by
\begin{equation}\label{linear-op}
    \begin{aligned}
&\mathcal{L}_1\left(\phi_1,\phi_2\right)\\:=&\phi_1-i_\lambda^{*}\Big(p\left(P_\lambda U_{\delta_1,\xi_1/\lambda}\right)^{p-1}\,\phi_1+\varepsilon\,\lambda^{2}\,\phi_1+\beta\lambda^{\frac{6-N}{2}} \,\mu_1^{-\frac{N-2}{4}}\,P_{\lambda}U_{\delta_1,\xi_1/\lambda}\,\phi_2\\&\quad\,\quad\,\quad\,+\beta\lambda^{\frac{6-N}{2}}  \mu_2^{-\frac{N-2}{4}}P_{\lambda}U_{\delta_2,\xi_2/\lambda}\,\phi_1\Big) \\
&\mathcal{L}_2\left(\phi_1,\phi_2\right)\\:=&\phi_2-i_\lambda^{*}\Big(p\left(P_{\lambda}U_{\delta_2,\xi_2/\lambda}\right)^{p-1}\,\phi_2+\varepsilon\,\lambda^{2}\,\phi_2+\beta\lambda^{\frac{6-N}{2}}\,\mu_1^{-\frac{N-2}{4}}\,P_{\lambda}U_{\delta_1,\xi_1/\lambda}\,\phi_2\\&\quad\,\quad\,\quad\,+\beta\lambda^{\frac{6-N}{2}}\mu_2^{-\frac{N-2}{4}}P_{\lambda}U_{\delta_2,\xi_2/\lambda}\,\phi_1\Big).
\end{aligned}
\end{equation}
Similarly, we define the error term $\bm{\mathcal{E}}_{\bm{\delta},\bm{\xi}}=\left(\mathcal{E}^{1}_{\bm{\delta},\bm{\xi}},\mathcal{E}^{2}_{\bm{\delta},\bm{\xi}} \right)$ with $\mathcal{E}^{i}_{\bm{\delta},\bm{\xi}}:=\Pi^{\perp}_i\mathcal{E}_i\in{K}^{\perp}_{i}$, $i=1,2$, where $\mathcal{E}_i$ is given by
\begin{equation}\label{error}
    \begin{aligned}
\mathcal{E}_1:= &i_\lambda^{*}\Big(\mu_1^{-\frac{N-2}{4}}\left(\left(P_\lambda U_{\delta_1,\xi_1/\lambda}\right)^{p}-U_{\delta_1,\xi_1/\lambda}^{p}\right)+\varepsilon\,\lambda^{2}\,\mu_1^{-\frac{N-2}{4}}\,P_{\lambda}U_{\delta_1,\xi_1/\lambda}\\&\quad\,\quad+\beta\lambda^{\frac{6-N}{2}}\mu_1^{-\frac{N-2}{4}}\mu_2^{-\frac{N-2}{4}}P_{\lambda}U_{\delta_1,\xi_1/\lambda}\,P_{\lambda}U_{\delta_2,\xi_2/\lambda}\Big), \\
\mathcal{E}_2:= & i_\lambda^{*}\Big(\mu_2^{-\frac{N-2}{4}}\left(\left(P_{\lambda}U_{\delta_2,\xi_2/\lambda}\right)^{p}-U_{\delta_2,\xi_2/\lambda}^{p}\right)+\varepsilon\,\lambda^{2}\,\mu_2^{-\frac{N-2}{4}}\,P_{\lambda}U_{\delta_2,\xi_2/\lambda}\\&\quad\,\quad+\beta\lambda^{\frac{6-N}{2}}\mu_1^{-\frac{N-2}{4}}\mu_2^{-\frac{N-2}{4}}\,P_{\lambda}U_{\delta_1,\xi_1/\lambda}P_{\lambda}U_{\delta_2,\xi_2/\lambda}\Big).
\end{aligned}
\end{equation}
Finally, we denote the nonlinear term $\bm{\mathcal{N}}_{\bm{\delta},\bm{\xi}}=\left(\mathcal{N}^{1}_{\bm{\delta},\bm{\xi}},\mathcal{N}^{2}_{\bm{\delta},\bm{\xi}} \right)$ with  $$\mathcal{N}^{i}_{\bm{\delta},\bm{\xi}}:=\Pi^{\perp}_i\mathcal{N}_i:{K}^{\perp}_{1}\times{K}^{\perp}_{2}\to{K}^{\perp}_{i},\quad i=1,2,$$
where $\mathcal{N}_i$ is defined by 
\begin{equation}
\label{perturbed}
\begin{aligned}
&\mathcal{N}_1\left(\phi_1,\phi_2\right)\\:=&i_\lambda^*\Bigg(\mu_1\left[\left(\mu_1^{-\frac{N-2}{4}}P_\lambda U_{\delta_1,\xi_1/\lambda}+\phi_1\right)^{p}-\left(\mu_1^{-\frac{N-2}{4}}P_\lambda U_{\delta_1,\xi_1/\lambda}\right)^{p}-p\,\mu_1^{-1}\left(P_\lambda U_{\delta_1,\xi_1/\lambda}\right)^{p-1}\,\phi_1\right]\\&\quad\,\quad+\beta\lambda^{\frac{6-N}{2}}\phi_1\,\phi_2\Bigg),\\
&\mathcal{N}_2\left(\phi_1,\phi_2\right)\\:=&i_\lambda^{*}\Bigg(\mu_2\left[\left(\mu_2^{-\frac{N-2}{4}}P_\lambda U_{\delta_2,\xi_2/\lambda}+\phi_2\right)^{p}-\left(\mu_2^{-\frac{N-2}{4}}P_{\lambda}U_{\delta_2,\xi_2/\lambda}\right)^{p}-p\,\mu_2^{-1}\left(P_{\lambda}U_{\delta_2,\xi_2/\lambda}\right)^{p-1}\,\phi_2\right]\\&\quad\,\quad+\beta\lambda^{\frac{6-N}{2}}\phi_1\,\phi_2\Bigg).
\end{aligned}
\end{equation}

\vspace{0.2cm}
\section{Proof of the Theorem (\ref{th1-4})}\label{three}

In the following we will assume that $\lambda:=\lambda(\varepsilon)\to 0$ as $\varepsilon\to0$
and
 \begin{equation}\label{suppose-beta}
       \left|\beta\right|= \begin{cases}
       o\left(\lambda|\ln{\lambda}|\right),\quad&N=4,\\
    o\left(\lambda\right),\quad&N=5,\\
    o\left(\lambda/\left|\ln{\lambda}\right|\right),\quad&N=6,\\
  o\left(\lambda^{{N-5}}\right)  ,\quad&N\geq7.\\
        \end{cases}
    \end{equation}

The   first step in the  Lyapunov-Schmidt procedure consist in solving problem \eqref{first-equ} in terms of $\left(\bm{\delta}, \bm{\xi}\right)$
as stated in the following Proposition whose proof is postponed in Appendix \ref{AAA}.

 \begin{Prop}\label{error-size}
 For any $\eta\in(0,1)$, there exists $C>0$ and $\varepsilon_0>0$ such that for any $\varepsilon \in (0,\varepsilon_0)$ and for any $\left(\bm{\delta}, \bm{\xi}\right)\in \mathcal{O}_\eta$, there exists a unique $\bm{\phi} \in K_{\bm{\delta},\bm{\xi}}^{\perp}$ verifying \eqref{first-equ} and
    \begin{align}\label{norm-geq5}
    \left\|\bm{\phi}\right\|_{H}\leq \begin{cases}
     C\,\varepsilon\lambda,\quad&N=4,\\
        C\,\varepsilon\lambda^{\frac{3}{2}},\quad&N=5,\\
        C\left(\lambda^{4}\left|\ln{\lambda}\right|+\varepsilon\,\lambda^{2}\left|\ln{\lambda}\right|\right),\quad&N=6,\\
            C\lambda^{\frac{N+2}{2}},\quad&N\geq 7.
    \end{cases}
\end{align}
\end{Prop} 

The second step in the Ljapunov-Schmidt procedure consists in solving  problem \eqref{second-equ}, which in virtue of Proposition \ref{error-size} can be rewritten as
\begin{align}\label{kker}
    \begin{cases}
 u-i_\lambda^*\left[\mu_{1} {g(u)}+\beta\lambda^{\frac{6-N}{2}} uv+\varepsilon\,\lambda^{2} u\right] =\sum\limits_{j=0}^{      N}c_j\,P_\lambda{\psi}^{j}_{\delta_1,\xi_1/\lambda},\\ v-i_\lambda^*\left[\mu_{2} {g(v)}+\beta\lambda^{\frac{6-N}{2}} uv+\varepsilon\,\lambda^{2}v\right] =\sum\limits_{j=0}^{N}d_j\,P_\lambda{\psi}^{j}_{\delta_2,\xi_2/\lambda}.
    \end{cases}
    \end{align}
The goal is finding suitable $(\bm{\delta}, \bm{\xi})\in\mathcal{O}_\eta$ such that  all the constants $c_j$ and $d_j$ are equal to zero, for $j=0,1,\ldots,N$, that is
\begin{align}\label{first-5}
    0=& \left\langle\mu_i^{-\frac{N-2}{4}}P_{\lambda}U_{\delta_i,\xi_i/\lambda}+\phi_i,P_{\lambda}\psi^{j}_{\delta_i,\xi_i/\lambda}\right\rangle_{X}\nonumber\\&-\left\langle i_\lambda^{*}\left(\mu_{i} {g\left(\mu_i^{-\frac{N-2}{4}}P_{\lambda}U_{\delta_i,\xi_i/\lambda}+\phi_i\right)}\right),P_{\lambda}\psi^{j}_{\delta_i,\xi_i/\lambda}\right\rangle_{X}\nonumber\\
    &-\left\langle i_\lambda^{*}\left(\varepsilon\lambda^{2}\left( \mu_i^{-\frac{N-2}{4}}P_{\lambda}U_{\delta_i,\xi_i/\lambda}+\phi_i\right)\right),P_{\lambda}\psi^{j}_{\delta_i,\xi_i/\lambda}\right\rangle_{X}\nonumber\\&-\left\langle i_\lambda^{*}\left(\beta\lambda^{\frac{6-N}{2}} \left(\mu_i^{-\frac{N-2}{4}}P_{\lambda}U_{\delta_i,\xi_i/\lambda}+\phi_i\right)\,\left(\sum_{j\neq i}\mu_j^{-\frac{N-2}{4}}P_{\lambda}U_{\delta_j,\xi_j/\lambda}+\phi_j\right)\right),P_{\lambda}\psi^{j}_{\delta_i,\xi_i/\lambda}\right\rangle_{X},
\end{align}
where $g(\omega)=\left(\omega^{+}\right)^{p}$, $p=2^*-1$.

    We would like to point out that,  the strategies for solving  \eqref{first-5} with $N=4$ is different from the higher dimensional cases $N\geq 5$, owing to the rate of the error term $\|\bm\phi\|_{H}$ obtained in Proposition \ref{error-size} is not small enough.  
    \\
    
 The proof when $N\geq5$ relies on the following Proposition whose proof is postponed in Appendix \ref{BBB}.
\begin{Prop}
    \label{supercritical} Let $N\geq 5$. For $i,l=1,2$ with $i\neq l $ and  if $ |\beta|$ verifies \eqref{suppose-beta}, then
\begin{equation}\label{7-0}
\begin{split}
    \text{RHS of \eqref{first-5}}=& {a}_i\,\Phi_\Omega(\xi_i)\,\lambda^{N-2}-{b}_i\,\varepsilon\,\lambda^{2}+o(\lambda^{N-2}).
    \end{split}
\end{equation}
Moreover, for $j=1,\cdots,N$, 
\begin{equation}\label{7-n}
\begin{split}
    \text{RHS of \eqref{first-5}}=&{t}_i\,\lambda^{N-1}\left(\frac{\partial }{\partial\xi_{i}^{j}}\Phi_\Omega(\xi_i)\right)\,+o(\varepsilon\lambda^{3})+o(\lambda^{N-1}).
    \end{split}
\end{equation}
Here the positive constants $ {a}_i$, $ {b}_i$ and $ {t}_i$ are given by
\begin{align*}
    {a}_i:=\frac{N-2}{2}\mu_{i}^{-\frac{N-2}{4}}A^2\delta_i^{N-3},\quad
   {b}_i:=\mu_{i}^{-\frac{N-2}{4}}\delta_i\,B,\quad {t}_i:=\mu_{i}^{-\frac{N-2}{4}}A^2\delta_i^{N-2},\quad i=1,2.
\end{align*}
\end{Prop}

    \begin{proof} [\textbf{Proof of  Theorem \ref{th1-4}}: completed]
     We argue exactly as in   \cite[Proof of Theorem 0.2]{musso}, taking into account Proposition \ref{error-size}.\end{proof}

\medskip

It only remains to study the case 
 $N=4$. We argue as in \cite{piro}.
Firstly, we present a sufficient condition which ensures that all the $c_j$ and $d_j$ in \eqref{kker} are zero. 
{\begin{Lem}\label{lem4.1}
    If the following identities hold true,
    \begin{equation} \label{4.2}
    \begin{cases}
       \displaystyle{\int_{\Omega_\lambda}}\left(-\Delta {u}-\varepsilon\,\lambda^{2} {u}-\mu_{1} {u}^{p}-\beta\,\lambda^{\frac{6-N}{2}} \,{u}\,{v}\right) P_\lambda{\psi}^{0}_{\delta_1,\xi_1/\lambda} dy= 0,\\
        \displaystyle{\int_{\Omega_\lambda}}\left(-\Delta {v}-\varepsilon\,\lambda^{2} {v}-\mu_{2} {v}^{p}-\beta\,\lambda^{\frac{6-N}{2}} \,{u}\,{v}\right) P_\lambda{\psi}^{0}_{\delta_2,\xi_2/\lambda} dy= 0, 
    \end{cases}     
    \end{equation}
    and for $k= 1,\ldots,4$,
   \begin{align}\label{4.3}
        \begin{cases}
        \int_{B_{\rho/\lambda}(\xi_1/\lambda)} \left(-\Delta {u}-\varepsilon\,\lambda^{2} {u}-\mu_{1} {u}^{p}-\beta\,\lambda^{\frac{6-N}{2}} \,{u}\,{v}\right)  \partial_k u\,dy = 0, \\
      \int_{B_{\rho/\lambda}(\xi_2/\lambda)} \left(-\Delta {v}-\varepsilon\,\lambda^{2} {v}-\mu_{2} {v}^{p}-\beta\,\lambda^{\frac{6-N}{2}} \,{u}\,{v}\right)  \partial_k v\,dy = 0,
    \end{cases}   
   \end{align} 
  where for some $\rho> 0$ and for any $y\in \Omega_\lambda$, the functions $\partial_{k}u(y):=\frac{\partial u(y)}{\partial y_k}$, $\partial_{k}v(y):=\frac{\partial v(y)}{\partial y_k}$. Then the constants $c_j=d_j= 0$, for any $j= 0, 1, \ldots, 4$.
\end{Lem}}
\begin{proof}
    The proof can be obtained by using similar arguments as in \cite[Lemma 4.1]{piro}. 
\end{proof}
\medskip
  {To address \eqref{4.3}, we introduce the following basic Lemma.}
  {\begin{Lem} \cite[Lemma 4.4]{piro}\label{lemtool}
    If there exists $C_1 > 0$ such that
    \begin{equation*}
    \int_{B(\xi,\eta_1) \setminus B(\xi,\eta_2)} |f(x)| \, dx \leq C_1,
    \end{equation*}
    then there exist $C_2 > 0$ and $\bar{\eta} \in (\eta_1, \eta_2)$ such that
    \begin{equation*}\label{eq:lemma44_2}
    \int_{\partial B(\xi,\bar{\eta})} |f(x)| \, d\sigma \leq C_2.
    \end{equation*}
    \end{Lem}}

\vspace{0.2cm}
Now, we compute the LHS of \eqref{4.2} and \eqref{4.3},  respectively.
\begin{Lem}\label{first}Let $N=4$. For $j=0$ and if ${|\beta|}$ verifies \eqref{suppose-beta}, then
    \begin{align}
&  \displaystyle{\int_{\Omega_\lambda}}\left(-\Delta {u}-\varepsilon\,\lambda^{2} {u}-\mu_{1} {u}^{p}-\beta\,\lambda^{\frac{6-N}{2}} \,{u}\,{v}\right) P_\lambda{\psi}^{0}_{\delta_1,\xi_1/\lambda} dy\nonumber\\
=&A_{1}\,\lambda^2\,\Phi_\Omega(\xi_1) -B_{1}\,\varepsilon\,\lambda^2\left|\ln{\lambda}\right|+o(\lambda^2)+o\left(\varepsilon\,\lambda^2\left|\ln{\lambda}\right|\right),\label{po1}\\
&  \displaystyle{\int_{\Omega_\lambda}}\left(-\Delta {v}-\varepsilon\,\lambda^{2} {v}-\mu_{2} {v}^{p}-\beta\,\lambda^{\frac{6-N}{2}} \,{u}\,{v}\right) P_\lambda{\psi}^{0}_{\delta_2,\xi_2/\lambda} dy\nonumber\\
=&A_{2}\,\lambda^2\,\Phi_\Omega(\xi_2) -B_{2}\,\varepsilon\,\lambda^2\left|\ln{\lambda}\right|+o(\lambda^2)+o\left(\varepsilon\,\lambda^2\left|\ln{\lambda}\right|\right),\label{po11}\end{align}
where the positive constants
$$A_{i}:=\mu_{i}^{-\frac{1}{2}}\,\left(\int_{\mathbb{R}^N}U_{1,0}^{3}\right)^{2}>0,\quad i=1,2,$$
and 
$ B_{i}>0$ is related to the integral $$C_4\,\delta_1\mu^{-\frac{1}{2}}_{i}\int_{B_{R/(\lambda\delta_i)}} \frac{|y|^2}{(1+|y|^2)^3}dy=C_4\,\delta_1\mu^{-\frac{1}{2}}_{i}\omega_3\,|\ln \lambda|+o\left(|\ln \lambda|\right),\quad \text{for some $R>0$}.$$

\end{Lem}
\begin{proof}
  We only show the estimates of \eqref{po1}, and \eqref{po11} can be derived by similar arguments. Note that \eqref{po1} is equivalent to \eqref{first-5} with $j=0$. Then using similar arguments as the terms $E_1$-$E_{10}$ given in Appendix \ref{BBB}, we derive that
\begin{equation}
    \label{M1}
    E_1=
        A_{1}\,\lambda^{2}\,\Phi_{\Omega}(\xi_1) +o\left(\lambda^{2}\right),
\end{equation}
and 
\begin{equation}
    \label{M2}
    E_2
    =-B_{1}\,\varepsilon\, \lambda^2|\ln{{\lambda}}|+o\left(\varepsilon\, 
  \lambda^2|\ln{{\lambda}}|\right).
   \end{equation}
For $E_3$, since 
\begin{align}\label{4-truth}
    \frac{\partial U_{\lambda\delta_1, \xi_1}(x)}{\partial \left(\lambda\delta_1\right)}\lesssim \frac{1}{\lambda}\,U_{\lambda\delta_1,\xi_1},
\end{align}
{then by the condition on $\beta$ given in \eqref{suppose-beta}, a direct computation yields that}
\begin{align}\label{M3}
\left|E_3\right|\lesssim
&\left|\beta\right|\,\displaystyle{\int_{\Omega}}\,U^2_{\lambda\delta_1,\xi_1}\,U_{\lambda\delta_2,\xi_2} dx\nonumber\\
=&\left|\beta\right|\,\left(\displaystyle{\int_{B_{\eta}(\xi_1)}}+\displaystyle{\int_{B_{\eta}(\xi_2)}}+\displaystyle{\int_{\Omega\backslash \left(B_{\eta}(\xi_1)\cup B_{\eta}(\xi_2)\right)}}\,U^2_{\lambda\delta_1,\xi_1}\,U_{\lambda\delta_2,\xi_2} dx\right)\nonumber\\
\lesssim&|\beta|\lambda^3|\ln{{\lambda}}|=o\left(\lambda^2\right).
\end{align}

\medskip
   Now taking into account the estimates of $\|\bm{\phi}\|_{H}$ obtained in Proposition \ref{error-size}, we compute the remaining terms. For $E_4$, by \eqref{varphi-1},
   \begin{align*}
       |E_{4}|\lesssim \left(\displaystyle{\int}_{\Omega}U^2_{\lambda\delta_1,\xi_1}\,\varphi^2_{\lambda\delta_1,\xi_1}\right)^{\frac{1}{2}}\|\bm{\phi}\|_{H}\lesssim\lambda^2|\ln{\lambda}|^{\frac{1}{2}}\|\bm{\phi}\|_{H}=o\left(\lambda^2\right).
   \end{align*}
 For $E_5$, by \eqref{varphi-2}, we have
\begin{align*}
    |E_5|\lesssim\lambda\left(\displaystyle{\int}_{\Omega}\left|P \psi_{\lambda\delta_1, \xi_1}^0-\psi_{\lambda\delta_1, \xi_1}^0\right|^4\right)^{\frac{1}{4}}\|\bm{\phi}\|_{H}\lesssim\lambda^2\|\bm{\phi}\|_{H}=\left(\lambda^2\right).
\end{align*}
For $E_6$, from \eqref{4-truth}, we get that
\begin{align*}
|E_6|\lesssim\varepsilon\,\lambda\left(\displaystyle{\int}_{\Omega}\left|\psi_{\lambda\delta_1, \xi_1}^0\right|^\frac{4}{3}\right)^{\frac{3}{4}}\|\bm{\phi}\|_{H}\lesssim\varepsilon\,\lambda\|\bm{\phi}\|_{H}=o\left(\lambda^2\right).
\end{align*}
Moreover, for $E_7$, 
\begin{align*}
    |E_7|\lesssim &|\beta|\lambda\,\|P_\lambda U_{\delta_2, \xi_2/\lambda}\,P_\lambda \psi_{\delta_1, \xi_1/\lambda}^0\|_{L^{\frac{4}{3}}(\Omega_\lambda)}\,\|\bm{\phi}\|_{H}\\
    \lesssim&|\beta|\,\left(\displaystyle{\int}_{\Omega}U^\frac{4}{3}_{\lambda\delta_2, \xi_2}(x)\,U^\frac{4}{3}_{\lambda\delta_1, \xi_1}(x)\,dx\right)^{\frac{3}{4}}\|\bm{\phi}\|_{H}\lesssim|\beta|\,\lambda^2\|\bm{\phi}\|_{H}=o\left(\lambda^2\right).
\end{align*}
Similarly, for $E_8$, we have
\begin{align*}
    |E_8|\lesssim &|\beta|\lambda\,\|P_\lambda U_{\delta_1, \xi_1/\lambda}\,P_\lambda \psi_{\delta_1, \xi_1/\lambda}^0\|_{L^{\frac{4}{3}}(\Omega_\lambda)}\,\|\bm{\phi}\|_{H}\\
    \lesssim&|\beta|\,\left(\displaystyle{\int}_{\Omega}U^\frac{8}{3}_{\lambda\delta_1, \xi_1}(x)\,dx\right)^{\frac{3}{4}}\|\bm{\phi}\|_{H}\lesssim|\beta|\,\lambda \|\bm{\phi}\|_{H}=o\left(\lambda^2\right).
\end{align*}

   In addition, we compute the nonlinear terms $E_9$ and $E_{10}$ involving the remainder term $\bm{\phi}$. Indeed,
   \begin{align*}
     |E_9|+|E_{10}|\lesssim  \|\bm{\phi}\|^2_{H}+|\beta|\left(\displaystyle{\int}_{\Omega}U^2_{\lambda\delta_1, \xi_1}(x)\right)^{\frac{1}{2}}\|\bm{\phi}\|^2_{H}=o\left(\lambda^2\right).
   \end{align*}

Thus, from \eqref{M1}, \eqref{M2}, \eqref{M3}  and the estimates of $E_4$-$E_{10}$, we derive \eqref{po1} holds. 

    \end{proof}
    
\vspace{0.2cm}

   In the following, we calculate the  local Pohozaev identities \eqref{4.3} to solve the reduced problem. Note that by the definition of $(u,v)$ given in \eqref{u-w}, \eqref{4.3} reduces to 
   \begin{equation}\label{prime}
        \begin{cases}
         \lambda\displaystyle\int_{B(\xi_1,\rho)} \left(-\Delta w_1 - w_1^3 - \varepsilon w_1 -\beta w_1\,w_2\right) \partial_j w_1\, dx = 0, \\
       \lambda\displaystyle \int_{B(\xi_2,\rho)} \left(-\Delta w_2 - w_2^3 - \varepsilon w_2 -\beta w_1\,w_2\right) \partial_j  w_2\,dx = 0,
    \end{cases}   
   \end{equation}
   for some \( \rho > 0 \). Moreover, we remark that $(u,v)$ solves \eqref{equivalent-lv} of the form
  $$ \left(u(y),v(y)\right)=\left(\mu_1^{-\frac{N-2}{4}} P_{\lambda} U_{\delta_1, \xi_1/\lambda}+\phi_1,\,\mu_2^{-\frac{N-2}{4}} P_{\lambda} U_{\delta_2, \xi_2/\lambda}+\phi_2 \right),\quad y\in\Omega_\lambda:=\Omega/\lambda,$$
   is equivalent to $(w_1,w_2)$ solves \eqref{lv} of the form
   \begin{equation}
      \label{formw}\left(w_1(x),w_2(x)\right)=\left(\mu_1^{-\frac{N-2}{4}} P U_{\lambda\delta_1, \xi_1}+\widehat{\phi}_1,\,\mu_2^{-\frac{N-2}{4}} P U_{\lambda\delta_2, \xi_2}+\widehat{\phi}_2 \right),\quad x\in\Omega,  
   \end{equation}
   where $\widehat{\phi}_i(x)=\lambda^{-1}{\phi}_i(x/\lambda)$ with $\|\widehat{\phi}_i\|_{H^{1}_0(\Omega)}=\|{\phi}_i\|_{H^{1}_0(\Omega_\lambda)}=O(\varepsilon\lambda)$, $i=1,2$.  

   \medskip
  Now we compute the following integral.
  { \begin{Lem}\label{newlemma}
Let $N=4$. For any \( \eta \in (0,1) \), there exists \( \rho > 0 \), such that for any \( (\bm{\delta},\bm{\xi}) \in {\mathcal{O}}_\eta \) and if $|\beta|$ satisfies \eqref{suppose-beta},  it holds that
        \begin{equation}\label{pohj-4}
           \int_{B(\xi_1,\rho)} \left(-\Delta w_1 - w_1^3 - \varepsilon w_1 -\beta w_1\,w_2\right) \partial_j w_1 
        = -l_1 \,\lambda^2  \,\partial_j \Phi_{\Omega}(\xi_1)  + o\left(\lambda^2\right), 
        \end{equation}
               
        \begin{align}\label{pohj'-4}
            \int_{B(\xi_2,\rho)} \left(-\Delta w_2 - w_2^3 - \varepsilon w_2 -\beta w_1\,w_2\right) \partial_j  w_2
            = -l_2 \,\lambda^2 \,  \partial_j \Phi_{\Omega}(\xi_2)  + o\left(\lambda^2\right),
            \end{align} where $l_i$, $i=1,2$, are positive constants.
        \end{Lem}}
        \begin{proof}
        We only prove \eqref{pohj-4} holds true and \eqref{pohj'-4} can be obtained analogously. Indeed, a direct computation yields that for some $\rho>0$,
        \begin{equation}\label{m1}
        \displaystyle\int_{B(\xi_1,\rho)} \left(-\Delta w_1 - w_1^3\right) \partial_j w_1\,dx
        = \displaystyle\int_{\partial B(\xi_1,\rho)} \left( -\partial_\nu w_1 \, \partial_j w_1 
        + \frac{1}{2} |\nabla w_1|^2 \nu_j - \frac{1}{4} w_1^4 \nu_j \right)\,dS_x.
        \end{equation}
 Note that $\forall\,x\in\partial B(\xi_1,\rho)$, $$U_{\lambda\delta_1, \xi_1}(x)=O(\lambda),\quad \frac{\partial U_{\lambda\delta_1, \xi_1}}{\partial x_{j}}=O(1).$$
Moreover, by Lemma \ref{lemtool} and Lemma \ref{error-size}, we choose suitable \( \rho >0\) such that
        \begin{equation}\label{m2}
        \displaystyle\int_{\partial B(\xi_1,\rho)} \left( |\nabla \widehat{\phi}_1|^2 + |\widehat{\phi}_1|^4 \right) \,dS_x= O(\varepsilon^2\lambda^2).
        \end{equation}
Furthermore, we note that
        \begin{equation}\label{m3}
        U_{\lambda\delta_1, \xi_1} := \frac{C_4\,\lambda\delta_1}{|x - \xi_1|^2} {- \widetilde C\lambda H(x, \xi_1)} + O(\lambda^2),
        \end{equation}
        uniformly in \( C^1 \) norm on \( \partial B(\xi_1,\rho) \),  where $C_4=2\sqrt2$, $\widetilde C=2C_4\omega\delta_1$, and $\omega$ denotes the measure of the unit sphere $S^3 \subset \mathbb{R}^4$. In addition, It is crucial to point out that the function \( H(x, \xi_1) \) is harmonic. Thus, by \eqref{m1}, \eqref{m2} and \eqref{m3}, we obtain that
        \begin{align*}
           &\displaystyle \int_{B(\xi_1,\rho)} \left(-\Delta w_1 - w_1^3 \right) \partial_j w_1\,dx \\
        =& \lambda^2 \int_{\partial B(\xi_1,\rho)} \,
        -\partial_\nu \left( \frac{C_4\delta_1}{|x - \xi_1|^2}- \widetilde C H(x, \xi_1) \right) 
        \partial_j \left( \frac{C_4\delta_1}{|x - \xi_1|^2} - \widetilde C H(x, \xi_1) \right) dS_x
        \\&+ \frac{1}{2} \lambda^2 \int_{\partial B(\xi_1,\rho)} \left| \nabla \left( \frac{C_4\delta_1}{|x - \xi_1|^2} - \widetilde C H(x, \xi_1) \right) \right|^2 \nu_j \,dS_x + o( \lambda^2)\\
        =&0+ C_4\delta_1\widetilde C \, \lambda^{2}\displaystyle\int_{\partial B(\xi_1,\rho)} \nabla\left( \frac{1}{|x - \xi_1|^2} \right) \cdot \nu \, \partial_{x_j} H(x,\xi_1)\,dS_x+ o( \lambda^2)\\
        =&\, -\frac{\widetilde{C}^2}{2}  \lambda^2\, \left(\frac{\partial }{\partial\xi_{1}^{j}}\Phi_\Omega(\xi_1)\right)+ o( \lambda^2),
        \end{align*}
 where the Robin's function \( \Phi_{\Omega}(x) = H(x,x) \) verifies
        \[
       \frac{\partial }{\partial\xi_{1}^{j}}\Phi_\Omega(\xi_1)= \left( \partial_{x_j} H(x,y) + \partial_{y_j} H(x,y) \right) \big|_{(x,y) = (\xi_1,\xi_1)} = 2 \left. \partial_{x_j} H(x,y) \right|_{(x,y) = (\xi_1,\xi_1)}.
        \]
      
        Thus, we conclude that
        \begin{equation}\label{mt1}
       \displaystyle \int_{B(\xi_1,\rho)} \left(-\Delta w_1- w_1^3 \right) \partial_j w_1\,dx = -\frac{\widetilde{C}^2}{2}  \lambda^2\, \left(\frac{\partial }{\partial\xi_{1}^{j}}\Phi_\Omega(\xi_1)\right)+ o( \lambda^2).
        \end{equation}

Similarly, we have
        \begin{align}\label{mt2}
        - \varepsilon\displaystyle\int_{B(\xi_1,\rho)}w_1 \,\partial_j w_1 \,dx
        &=
        - \frac{\varepsilon}{2}\displaystyle \int_{\partial B(\xi_1,\rho)}  w_1^2 \,\nu_j \,dS_x=O\left(\varepsilon\lambda^2\right)=o( \lambda^2).
        \end{align}

       Finally, we aim to show that
        \begin{align} \label{mt3}
            -\beta\displaystyle\int_{B(\xi_1,\rho)}  w_1\,w_2 \, \partial_j w_1\,dx
            =  o( \lambda^2).
            \end{align}
    Indeed, from \eqref{formw}, a direct computation yields that
        \begin{align*} 
        &\beta\displaystyle\int_{B(\xi_1,\rho)}  w_1\,w_2 \, \partial_j w_1 \,dx\\
            =&\underbrace{\beta\mu_1^{-\frac{1}{2}}\mu_2^{-\frac{1}{2}}\displaystyle \int_{B(\xi_1,\rho)}  PU_{ \lambda\delta_1,\xi_1}\,PU_{ \lambda\delta_2,\xi_2}\,\left(\partial_{x_j}PU_{ \lambda\delta_1,\xi_1}+\partial_{x_j}\widehat{\phi}_1\right)dx}_{:=J_1}\\
            &+\underbrace{\beta\mu_1^{-\frac{1}{2}}\displaystyle \int_{B(\xi_1,\rho)}  PU_{ \lambda\delta_1,\xi_1}\,\widehat{\phi}_2\,\left(\partial_{x_j}PU_{ \lambda\delta_1,\xi_1}+\partial_{x_j}\widehat{\phi}_1\right)dx}_{:=J_2}\\
            &+\underbrace{\beta\mu_2^{-\frac{1}{2}}\displaystyle \int_{B(\xi_1,\rho)}  \widehat{\phi}_1\,PU_{ \lambda\delta_2,\xi_2}\,\left(\partial_{x_j}PU_{ \lambda\delta_1,\xi_1}+\partial_{x_j}\widehat{\phi}_1\right)dx}_{:=J_3}\\
            &+\underbrace{\beta\displaystyle \int_{B(\xi_1,\rho)}  \widehat{\phi}_1\,\widehat{\phi}_2\,\left(\partial_{x_j}PU_{ \lambda\delta_1,\xi_1}+\partial_{x_j}\widehat{\phi}_1\right)dx}_{:=J_4},
            \end{align*}
where by \eqref{4-truth} and the assumptions on $\beta$ given in \eqref{suppose-beta}, we derive that
\begin{align*}
|J_1|\lesssim&|\beta|\displaystyle \int_{B(\xi_1,\rho)} U^2_{\lambda\delta_1,\xi_1}dx+ |\beta|\left(\displaystyle \int_{B(\xi_1,\rho)} U^2_{\lambda\delta_1,\xi_1}U^2_{\lambda\delta_2,\xi_2}dx\right)^{\frac{1}{2}}\|\nabla \widehat{\phi}\|_2\\\lesssim  & |\beta|\lambda^2|\ln{\lambda}|+|\beta|\lambda^2|\ln{\lambda}|^{\frac{1}{2}}\left(\varepsilon\lambda\right)=o(\lambda^2),
\end{align*}
and
            \begin{align*}
                |J_2|                \lesssim&|\beta|\left\|U_{\lambda\delta_1,\xi_1}\right\|_4\,\left\|\nabla U_{\lambda\delta_1,\xi_1}\right\|_2\,\|\widehat{\phi}_2\|_4+|\beta|\left\|U_{\lambda\delta_1,\xi_1}\right\|_4\,\|\nabla \widehat{\phi}_2\|_2\,\|\widehat{\phi}_2\|_4\\
                \lesssim&|\beta|(\varepsilon\lambda)=o(\lambda^2).
            \end{align*}
            Moreover, 
            \begin{align*}
                |J_3|\lesssim&|\beta|\displaystyle \int_{B(\xi_1,\rho)}  U_{\lambda\delta_2,\xi_2}\,|\widehat{\phi}_1|\,\left(|\partial_{x_j}PU_{\lambda\delta_1,\xi_1}|+|\partial_{x_j}\widehat{\phi}_1|\right)dx\\
                \lesssim&|\beta|\lambda\,\left\|\nabla U_{\lambda\delta_1,\xi_1}\right\|_2\,\|\widehat{\phi}\|_2+|\beta|\lambda\,\|\nabla \widehat{\phi}_1\|_2\|\widehat{\phi}_1\|_2\\
                \lesssim&|\beta|\lambda\,(\varepsilon\lambda)=o(\lambda^2),
            \end{align*}
            and 
            \begin{align*}
|J_4|\lesssim&|\beta|\|\widehat{\phi}_1\|_{H^{1}_0(\Omega)}\,\|\widehat{\phi}_2\|_{H^{1}_0(\Omega)}\lesssim|\beta|(\varepsilon^2\lambda^2)=o(\lambda^2).
            \end{align*}
Thus, from the estimates of $J_1$-$J_4$, we derive that {if $|\beta|$} satisfies \eqref{suppose-beta}, then \eqref{mt3} holds true. Furthermore, from \eqref{mt1}, \eqref{mt2} and \eqref{mt3}, we deduce \eqref{pohj-4}. Using similar arguments to $v$, we derive \eqref{pohj'-4}. This completes the proof.
            \end{proof}

\vspace{0.2cm}
 \begin{proof} [\textbf{Proof of  Theorem \ref{th1-4}}: completed]  We  choose
    \begin{align*}
        \lambda=e^{-\frac{d_i}{\varepsilon}},\quad \text{for some $d_i>0$, $i=1,2$.}
    \end{align*}
 and  the proof follows arguing as in  \cite[Proof of Theorem 1.1]{Pistoia2017}.
    \end{proof}
   
     \vspace{0.2cm}
%\section*{Appendix}
%\appendix
%\renewcommand{\theequation}{A.\arabic{equation}}
\appendix

\section{Solving problem (\ref{first-equ})} \label{AAA}
\setcounter{equation}{0}
\renewcommand{\theequation}{A.\arabic{equation}}
First of all, we study the linear theory associated with problem  \eqref{first-equ}.

\begin{Lem}\label{invert}
    For any $\eta\in(0,1)$, there exists $\varepsilon_0>0$ such that for any $\varepsilon \in (0,\varepsilon_0)$, for any $\left(\bm{\delta}, \bm{\xi}\right)\in \mathcal{O}_\eta$, and $$|\beta|=o\left(\lambda^{\frac{N-6}{2}}\right)\to0,\quad \text{if $N\geq 6$, (no assumption is needed if $N=4,5$),}$$  it holds true that  
    \begin{align}\label{5-invert-1}
        \left\|\bm{\mathcal{L}}_{\bm{\delta},\bm{\xi}}\left(\bm{\phi}\right)\right\|_{H}\geq c\left\|\bm{\phi}\right\|_{H}, \quad \forall \bm{\phi}\in {K}^{\perp}_{1}\times{K}^{\perp}_{2},    \end{align} 
   for some constants $c>0$.
\end{Lem}
\begin{proof}  

Suppose by contradiction that, for any $\eta\in(0,1)$ with $\left(\bm{{\bar{\delta}}},\bm{\xi} \right)\in \mathcal{O}_\eta$ and $n$ large enough, there exist:
   \begin{itemize}
    \item sequences $\left(\bm{{\delta}}_n,\bm{\xi}_n \right)\in \mathcal{O}_\eta$, such that $\xi_{{in}} \rightarrow \xi_i  $ \text{and} $\delta_{{in}}\to\bar{\delta}_i$, $i=1,2$.
    \item sequences $\varepsilon_{{n}} \to 0$ and  $\lambda_n=\lambda(\varepsilon_n)\to 0$.
    \item sequences $\beta_n$ such that $|\beta_n|\lambda_n^{\frac{6-N}{2}}\longrightarrow 0$ if $N\geq 5$, and $|\beta_n|\lambda_n|\ln{\lambda_n}|^{\frac{1}{2}}\longrightarrow 0$ if $N=4$.
    \item  sequences $\phi_{{in}} \in K_i^{\perp}$, $i=1,2$, such that $\left\|\bm{\phi_n}\right\|_{H}=\left\|{\phi_{1n}}\right\|_{X}+\left\|{\phi_{2n}}\right\|_{X}=1$ and \end{itemize} 
    \begin{align*}
\left\|\bm{\mathcal{L}}_{\bm{\delta_n},\bm{\xi_n}}\left(\bm{\phi_n}\right)\right\|_{H}&=\left\|\left(\mathcal{L}^{1}_{\bm{\delta_n},\bm{\xi_n}}\left(\bm{\phi_n}\right),\mathcal{L}^{2}_{\bm{\delta_n},\bm{\xi_n}} \left(\bm{\phi_n}\right)\right)\right\|_{H}\\&:=\left\|\left({h}_{1n},{h}_{2n}\right)\right\|_{{H}} \rightarrow 0,\quad\text{as $n\rightarrow +\infty$},
    \end{align*}
From \eqref{linear-op}, we find that for $i=1,2$, 
\begin{align}\label{omega-in}
   \phi_{in}=&\Pi^{\perp}_i\circ i_{\lambda_n}^{*}\Big(p\left(P_{\lambda_n} U_{in}\right)^{p-1}\,\phi_{in}+\varepsilon_n\,\lambda_n^{2}\,\phi_{in}+\beta_n\,\lambda_n^{\frac{6-N}{2}} \,\mu_i^{-\frac{N-2}{4}}\,P_{{\lambda_n}}U_{in}\sum_{j\neq i}\phi_{jn}\nonumber\\&\quad\,\quad\,+\beta_n\,\lambda_n^{\frac{6-N}{2}}\phi_{in} \sum_{j\neq i}\mu_j^{-\frac{N-2}{4}}P_{{\lambda_n}}U_{jn}\Big) + h_{in}+ w_{in},
\end{align}
where $U_{{in}}:=U_{\delta_{{in}}, \xi_{{in}}/\lambda_n}$, $w_{{in}} \in K_i$ and $
h_{in}\in {K}^{\perp}_{i}$ with $\left\|{h}_{in}\right\|_{{X}} \rightarrow 0$ {as $n\rightarrow +\infty$}.

\medskip
\textbf{Step 1.}~ We prove that $w_{{in}} \rightarrow 0$ strongly in $X$ as $n \rightarrow+\infty$.

Since $w_{{in}}\in K_i$, then $w_{{in}}=\sum\limits_{k=0}^{N} c_{in}^k\,P_{{\lambda_n}}\psi_{{in}}^k$, for some coefficients $\{ c_{in}^k \}$ and $\psi_{{in}}^k:=\psi_{\delta_{{in}},\, \xi_{{in}}/\lambda_n}^k$. Thus,  we obtain that
\begin{align}\label{w-in-1}
\|w_{in}\|^2_{X} =\sum_{k,j=0}^{N} c_{in}^{k}c_{in}^{j}\, \left\langle P_{\lambda_n} \psi_{{in}}^j, P_{\lambda_n} \psi_{{in}}^k\right\rangle _{X}= \sum_{k=0}^{N} \left(c_{in}^{k}\right)^2\, \sigma_{kk} + \sum_{\substack{j,k=0 \\ j\ne k}}^{N} c_{in}^{k}c_{in}^{j}\,  o_n(1),
\end{align}
 where
\begin{align*}
\displaystyle \sigma_{00}
&=p\, C_N^{2^*}\, \frac{(N-2)^2}{4}\,\displaystyle\int_{\R^N}\frac{ (| y|^2 -1)^2}
{(1 + | y|^2)^{N+2}}\,dy>0,\\
\sigma_{jj}  &=p\,C_N^{2^*}\, (N - 2)^2 \,\displaystyle\int_{\R^N}\frac{y_j^2}{(1 + | y|^2)^{N+2}}\,dy>0, 
\quad j = 1, \dots, N. 
\end{align*}
On the other hand, from \eqref{omega-in}, we have
\begin{align}\label{w-in-2}
         &\left\| w_{{in}} \right\|_{X}^2\nonumber\\= &- \underbrace{p\,\displaystyle\int_{\Omega_{\lambda_n}} \left(P_{{\lambda_n}}U_{in}\right)^{p-1} \,\phi_{in} \left(\sum\limits_{k=0}^{N} c_{in}^k\left(P_{{\lambda_n}}\psi_{{in}}^k-\psi_{{in}}^k\right)\right) \, dx}_{:={I}_n}\nonumber\\&-\underbrace{p\,\displaystyle\int_{\Omega_{\lambda_n}} \left[\left(P_{\lambda_n}U_{in}\right)^{p-1} -U_{in}^{p-1}\right]\,\phi_{in} \left(\sum\limits_{k=0}^{N} c_{in}^k\,\psi_{{in}}^k\right) \, dx}_{:=II_n}\nonumber\\
& \underbrace{-\beta_n \lambda_n^{\frac{6-N}{2}}\int_{\Omega_{\lambda_n}} \left(\mu_i^{-\frac{N-2}{4}}P_{\lambda_n}U_{in}\sum_{j\neq i}\phi_{jn} +\phi_{in}\sum_{j\neq i}\mu_j^{-\frac{N-2}{4}}P_{\lambda_n}U_{jn}\right)\left(\sum\limits_{k=0}^{N} c_{in}^k\,P_{{\lambda_n}}\psi_{{in}}^k\right) dx}_{:={III}_{{n}}}\nonumber\\&\underbrace{-\varepsilon_n\,\lambda_n^{2}\displaystyle\int_{\Omega_{\lambda_n}} \phi_{{in}}\,\left(\sum\limits_{k=0}^{N} c_{in}^k\,P_{{\lambda_n}}\psi_{{in}}^k\right)  dx}_{:=IV_n}+\left\langle\phi_{{in}}, w_{{in}}\right\rangle _{X}-\left\langle h_{{in}}, w_{{in}}\right\rangle _{X} .   \end{align}
Obviously, $\left\langle h_{{in}}, w_{{in}}\right\rangle _{X}=o_{n}(1)$ since $ h_{{in}}\to 0$ strongly in $X$. Moreover, since $\phi_{{in}}\in K^{\perp}_i$ and $w_{{in}}\in K_i$, then
$
\left\langle\phi_{{in}}, w_{{in}}\right\rangle _{X} =0
$. Furthermore, from Lemma \ref{1}, Lemma \ref{esimates} and Lemma \ref{7}, we have
\begin{align*}
    |I_n|+|II_n|\lesssim o_n(1) \sum_{k=0}^N\,\left|c_{{in}}^k\right|,\quad \left|IV_n\right|\lesssim o_{n}(1)\sum_{k=0}^N \left|c_{{in}}^k\right|.
\end{align*}
In addition, for $N\geq 4$,
\begin{align*}
|III_n|
\lesssim
    |\beta_n|\lambda_n^{\frac{6-N}{2}}\sum_{k=0}^N \left|c_{{in}}^k\right|.
\end{align*} 
Hence, combining \eqref{w-in-2} with the above estimates, we conclude that if {$|\beta_n|=o\left(\lambda_n^{{(N-6)/2}}\right)\to 0$ as $n\to+\infty$ when $N\geq 6$ (no assumption is needed if $N=4,5$)}, then
\begin{align}
    \left\| w_{\mathrm{in}} \right\|_{H_0^1(\Omega_{\lambda_n})}^2
\lesssim\,o_n(1)\left\|w_{{in}}\right\|_{H_{0}^{1}(\Omega_{\lambda_n})}+o_n(1)\left(\sum_{k=0}^{N} \left|c_{{in}}^k\right|\right),\label{w-in-3}
    \end{align}
then from \eqref{w-in-1}, we obtain that  
  \begin{align*}
        \left\| w_{{in}} \right\|_{H_0^1(\Omega_{\lambda_n})}^2=&\sum_{k=0}^N\left(c_{{in}}^k\right)^2 \sigma_{kk}+{o}_n(1)\sum_{\substack{j, k=0 \\ j \neq k}}^4 \left|c_{{in}}^j c_{i n}^k\right| =o_n(1)<+\infty,
    \end{align*}
which implies that {the sequences $\{c_{{in}}^k\}$ are bounded in $n$ and small enough}, and $w_{{in}}\rightarrow 0$ strongly in $X$ as $n\rightarrow+\infty$.

Now we consider the open set $\widetilde{\Omega}_{\lambda_{in}}:=\Omega_{\lambda_n}-\xi_{in} / \lambda_n$ and the function
$$
\widetilde{\phi}_{in}(y):=\phi_n\left(y+\xi_{in} / \lambda_n\right), \quad y \in \widetilde{\Omega}_{\lambda_{in}},\quad i=1,2.
$$
Obviously, $\left\|\widetilde{\phi}_{in}\right\|_{\widehat{X}}=\left\|\phi_{in}\right\|_{X} \leq  \left\|\bm{\phi_n}\right\|_{H}=1$, where 
\begin{align*}
    \widehat{X}:=\begin{cases}
        \mathcal{D}^{1,2}\left(\mathbb{R}^N\right) \cap L^s\left(\mathbb{R}^N\right),\quad&\text{for $N\geq 7$,}\\\mathcal{D}^{1,2}\left(\mathbb{R}^N\right),\quad&\text{for $N=4,5,6$} .
    \end{cases}
\end{align*}Then, up to a subsequence, we can assume that there exists $\widetilde{\phi}_i \in \widehat{X}$ such that
$$
\widetilde{\phi}_{in} \rightharpoonup \widetilde{\phi}_i,\quad \text { weakly in $\widehat{X}$},
$$

\medskip

\textbf{Step 2.}~ We prove that \( \widetilde{\phi}_{in}\rightharpoonup 0 \) weakly in \( \widehat{X}\).

We only prove the case $i=1$. Indeed, from \eqref{omega-in},  a direct computation yields that
\begin{align}\label{translation}
    \widetilde{\phi}_{1n} = &\Pi^{\perp}_1\circ i_{\widetilde{\lambda}_n}^{*}\Big(p\left(P_{\widetilde{\Omega}_{\lambda_{1n}}} U_{\delta_{1n},0}\right)^{p-1}\,\widetilde\phi_{1n}+\varepsilon_n\,\lambda_n^{2}\,\widetilde\phi_{1n}\nonumber\\&\quad\,\quad\,\quad\,\quad+\beta_n\lambda_n^{\frac{6-N}{2}} \,\mu_1^{-\frac{N-2}{4}}\,P_{\widetilde{\Omega}_{\lambda_{1n}}}U_{\delta_{1n},0}\,\phi_{2n}\left(y+\xi_{1n} / \lambda_n\right)\nonumber\\&\quad\,\quad\,\quad\,\quad+\beta_n\lambda_n^{\frac{6-N}{2}}\widetilde{\phi}_{1n}(y)\,\mu_2^{-\frac{N-2}{4}}P_{{\lambda_n}}U_{\delta_{2n},\xi_{2n}/\lambda_n}\left(y+\xi_{1n} / \lambda_n\right)\Big)\nonumber\\&+ \widetilde{h}_{1n} + \widetilde{w}_{1n},
\end{align}
where
\begin{equation*}
\widetilde{w}_{1n}(y) :=w_{in}\left(y+\xi_{1n} / \lambda_n\right), \quad 
\widetilde{h}_{1n}(y) := h_{1n}\left(y+\xi_{1n} / \lambda_n\right),\quad y\in \widetilde{\Omega}_{\lambda_{1n}}.
\end{equation*}
Take $\Phi \in C_c^\infty(\mathbb{R}^N)$ and $n$ large enough, such that 
$K_\Phi:= \operatorname{supp} \Phi \subseteq \widetilde{\Omega}_{\lambda_{1n}}$. Then we have
\begin{equation*}
\begin{aligned}
  &\displaystyle\int_{\widetilde{\Omega}_{\lambda_{1n}}} \nabla \widetilde{\phi}_{1n}(y)\, \nabla \Phi(y)  dy\\= & \underbrace{p  \displaystyle\int_{\widetilde{\Omega}_{\lambda_{1n}}}\left(P_{\widetilde{\Omega}_{\lambda_{1n}}} U_{\delta_{1n},0}\right)^{p-1}\,\widetilde\phi_{1n} \,\Phi(y) dy}_{:=A_n}\\&
+\varepsilon_n\,\lambda_n^{2}\displaystyle\int_{\widetilde{\Omega}_{\lambda_{1n}}} \widetilde\phi_{1n}\,\Phi(y) dy+\displaystyle \int_{\widetilde{\Omega}_{\lambda_{1n}}} \nabla\left(\widetilde{h}_{1n} + \widetilde{w}_{1n}\right) \nabla \Phi(y) dy \\&
+\underbrace{\beta_n\lambda_n^{\frac{6-N}{2}} \,\mu_1^{-\frac{N-2}{4}}\displaystyle\int_{\widetilde{\Omega}_{\lambda_{1n}}} P_{\widetilde{\Omega}_{\lambda_{1n}}}U_{\delta_{1n},0}(y)\,\phi_{2n}\left(y+\xi_{1n} / \lambda_n\right) \,\Phi(y) dy}_{:=B_n}\\&
 + \underbrace{\beta_n\lambda_n^{\frac{6-N}{2}}\,\mu_2^{-\frac{N-2}{4}}\displaystyle\int_{\widetilde{\Omega}_{\lambda_{1n}}} P_{{\lambda_n}}U_{\delta_{2n},\xi_{2n}/\lambda_n}\left(y+\xi_{1n} / \lambda_n\right)\,\widetilde{\phi}_{1n}\left(y\right) \,\Phi(y) dy}_{:=C_n}.
\end{aligned}
\end{equation*}
It is easy to get that
\[
\displaystyle\int_{\widetilde{\Omega}_{\lambda_{1n}}} \nabla(\widetilde{h}_{in} + \widetilde{w}_{in}) \,\nabla \Phi \to 0,
\]
\begin{equation*}
\begin{aligned}
A_n& =p\displaystyle\int_{\widetilde{\Omega}_{\lambda_{1n}}}
U_{\delta_{1n},0}^{p-1}(y)\,\widetilde\phi_{1n}(y) \,\Phi(y)\,dy + o_n(1)\\&
\to p\,\displaystyle\int_{\mathbb{R}^N} U_{\bar{\delta}_{1},0}^{p-1}(y)\,\widetilde{\phi}_1(y)\,\Phi(y)\,dy,\quad\text{for some fixed $\bar{\delta}_{1}$}. 
\end{aligned}
\end{equation*}
Moreover,  when $N\geq 4$, if ${\left|\beta_n\right|\,\lambda_n^{\frac{6-N}{2}} \to 0}$ as $n\to+\infty$, then
\begin{equation*}
    \begin{aligned}
   \left|B_n\right|+\left|C_n\right| 
    &\lesssim \left|\beta_n\right|\,\lambda_n^{\frac{6-N}{2}} \, \|\Phi\|_{L^{\frac{N}{2}}(\R^N)}\,\|\phi_{jn}\|_{L^{{2^*}}(\R^N)}\,\|U_{\delta_{1n},0}\|_{L^{2^*}(\R^N)}\rightarrow0.
    \end{aligned}
    \end{equation*} 
 Thus,  we find that $ \widetilde{\phi}_1\in \mathcal{D}^{1,2}(\mathbb{R}^N)$ solves
\begin{equation}
      \label{tilder-phi}-\Delta \widetilde{\phi}_1 =p\,  U_{\bar{\delta}_{1},0}^{p-1} \,\widetilde{\phi}_1 \quad \text{in } \mathbb{R}^N.
\end{equation}

Next, we check that \( \widetilde{\phi}_1 \equiv 0 \). For that, it is enough to show that $$ \widetilde{\phi}_1 \in \left(\ker\left(-\Delta - p\,  U_{\bar{\delta}_{1},0}^{p-1}  \right)\right)^\perp ,$$ that is,
\[
\displaystyle\int_{\mathbb{R}^N} \nabla \widetilde{\phi}_1 \,\nabla \psi_1^{j} = 0, \quad \forall j = 0,1, \ldots, N,
\]
where $$\psi_1^{0}=\frac{\partial U_{\bar{\delta}_{1},0}}{\partial\bar{\delta}_{1}},\quad \psi_1^{j}=\frac{\partial U_{\bar{\delta}_{1},\bar{\xi}_1}}{\partial\bar{\xi}_{1}^{j}}\Bigg|_{\bar{\xi}_1=0},\quad \forall j = 1, \ldots, N. $$
Indeed, for \( j = 0 \), since ${\phi}_{1n}\in K_{1}^\perp$ given in \eqref{space}, we have
\begin{equation*}
\begin{aligned}
0 = p\,\displaystyle\int_{{\Omega}_{\lambda_{n}}}U^{{p-1}}_{\delta_{1n},\xi_{1n}/\lambda_n}\,\psi_{1n}^0\,\phi_{1n}=p\,\displaystyle\int_{\widetilde{{\Omega}}_{\lambda_{1n}}} \,U^{{p-1}}_{\delta_{1n},0}\,\frac{\partial U_{\delta_{1n},0}}{\partial \delta_{1n}}\,\widetilde{\phi}_{1n}.
\end{aligned}
\end{equation*}
Moreover, as \( \widetilde{\phi}_{1n}\rightharpoonup \widetilde{\phi}_1 \) in \( L^\frac{2N}{N-2}(\mathbb{R}^N) \) and \( U^{{p-1}}_{\delta_{1n},0}\,\frac{\partial U_{\delta_{1n},0}}{\partial \delta_{1n}}\in L^{\frac{2N}{N+2}}(\mathbb{R}^N) \), then passing to the limit, we obtain that
\[
p\,\displaystyle\int_{\mathbb{R}^N}\,U_{\bar{\delta}_{1},0}^{p-1}\, \frac{\partial U_{\bar{\delta}_{1},0}}{\partial\bar{\delta}_{1}}\,\widetilde{\phi}_{i}=0.
\]
The proof for \( j = 1, \dots, N\) is analogous. Thus, combining with  \eqref{tilder-phi}, we conclude that $\widetilde{\phi}_1 \equiv 0 $, which  implies that \( \widetilde{\phi}_{1n}\rightharpoonup 0 \) weakly in \( \mathcal{D}^{1,2}(\mathbb{R}^N) \). Furthermore, since the uniqueness of the limit function $\widetilde{\phi}_1$, we derive that   \( \widetilde{\phi}_{1n}\rightharpoonup 0 \) weakly in \( L^{s}(\mathbb{R}^N) \) if $N\geq 7$.
\medskip

\textbf{Step 3.}~We prove that $\phi_{in} \to 0$ strongly in $X$, for $i = 1,2$. In fact, from \eqref{translation}, we get that if ${|\beta_n|\,\lambda_n^{\frac{6-N}{2}} \to 0}$ as $n \to+\infty$, then
\begin{align*}
  &\displaystyle\int_{\widetilde{\Omega}_{\lambda_{1n}}} \left|\nabla \widetilde{\phi}_{1n}\right|^2 \\= & {p  \displaystyle\int_{\widetilde{\Omega}_{\lambda_{1n}}}\left(P_{\widetilde{\Omega}_{\lambda_{1n}}} U_{\delta_{1n},0}\right)^{p-1}\,\left|\widetilde\phi_{1n}\right|^2 }\\&
+\varepsilon_n\,\lambda_n^{2}\displaystyle\int_{\widetilde{\Omega}_{\lambda_{1n}}} \left|\widetilde\phi_{1n}\right|^2+\displaystyle \int_{\widetilde{\Omega}_{\lambda_{1n}}} \nabla\left(\widetilde{h}_{1n} + \widetilde{w}_{1n}\right) \nabla \widetilde\phi_{1n}(y)  \\&
+\beta_n\lambda_n^{\frac{6-N}{2}} \,\mu_1^{-\frac{N-2}{4}}\displaystyle\int_{\widetilde{\Omega}_{\lambda_{1n}}} P_{\widetilde{\Omega}_{\lambda_{1n}}}U_{\delta_{1n},0}(y)\,\phi_{2n}\left(y+\xi_{1n} / \lambda_n\right) \,\widetilde\phi_{1n} \\&
 + {\beta_n\lambda_n^{\frac{6-N}{2}} \,\mu_2^{-\frac{N-2}{4}}\displaystyle\int_{\widetilde{\Omega}_{\lambda_{1n}}} P_{{\lambda_n}}U_{\delta_{2n},\xi_{2n}/\lambda_n}\left(y+\xi_{1n} / \lambda_n\right)\,\left|\widetilde\phi_{1n}\right|^2 }\\\lesssim&\begin{cases}
\varepsilon_n\,\lambda_n^{2}+\left|\beta_n\right|\lambda_n^{\frac{6-N}{2}}+o_{n}(1),\quad&\text{$N\geq 5$,}\\
\varepsilon_n\,\lambda_n^{2}+\left|\beta_n\right|\lambda_n^{{2}}|\ln{\lambda_n}|^{\frac{1}{2}}+o_{n}(1),\quad&\text{$N=4$,}\\
 \end{cases}\\=&o_{n}(1),
\end{align*}
similarly, 
\begin{align*}
    \displaystyle\int_{\widetilde{\Omega}_{\lambda_{2n}}} \left|\nabla \widetilde{\phi}_{2n}\right|^2=o_{n}(1).
\end{align*}

Thus, we conclude that if {$|\beta_n|=o\left(\lambda_n^{{(N-6)/2}}\right)\to 0$ as $n\to+\infty$ when $N\geq 6$ (no assumption is needed if $N=4,5$)}, then
\begin{align}
    \label{li-0}\lim _{n \rightarrow \infty} \displaystyle\int_{{\Omega}_{\lambda_{n}}}\left|\nabla {\phi}_{in}\right|^2=0,
\end{align}
which implies that ${\phi}_{in}\to 0$ strongly in $H_{0}^{1}({\Omega}_{\lambda_{n}})$, $i=1,2$.
In particular, when $N\geq 7$, we need to prove that ${\phi}_{in}\to 0$ strongly in $L^{s}({\Omega}_{\lambda_{n}})$. Indeed, from Lemma \ref{ls} and \eqref{omega-in}, we derive that
\begin{align*}
    \left\|{\phi}_{1n} \right\|_{s}\lesssim& \| U_{1n}^{p-1}\,\phi_{1n}\|_{{\frac{N s}{N+2 s}}}+\varepsilon_n\,\lambda_n^{2}\,\|\phi_{1n}\|_{{\frac{N s}{N+2 s}}}+\|{h}_{1n} \|_{{\frac{N s}{N+2 s}}}+\|{w}_{1n}\|_{{\frac{N s}{N+2 s}}}\nonumber\\&+\left|\beta_n\right|\,\lambda_n^{\frac{6-N}{2}} \|U_{{1n}}\,\phi_{2n}\|_{{\frac{N s}{N+2 s}}}+\left|\beta_n\right|\,\lambda_n^{\frac{6-N}{2}} \|U_{{2n}}\,{\phi}_{1n}\|_{{\frac{N s}{N+2 s}}}.
\end{align*}
Since $\frac{N s}{N+2 s}\in(2,s)$, then by the interpolation formula and \eqref{li-0}, it is easy to check that
\begin{align*}
   \|\phi_{1n}\|_{{\frac{N s}{N+2 s}}}\lesssim \|\phi_{1n}\|^{\theta}_{s}\,\|\phi_{1n}\|^{1-\theta}_{2}\lesssim \|\nabla\phi_{1n}\|^{1-\theta}_{2}=o_{n}(1),
\end{align*}
where $ \theta\in(0,1)$ with $\frac{\theta}{s}+\frac{1-\theta}{2}=\frac{N+2 s}{N s}$. Hence,  we derive that ${\phi}_{1n}\to 0$ strongly in $L^{s}({\Omega}_{\lambda_{n}})$ provided $|\beta_n|=o\left(\lambda_n^{{(N-6)/2}}\right)\to 0$ as $n\to+\infty$ when $N\geq 6$, (no assumption is needed if $N=4,5$). Analogously, for $i=2$ we obtain similar estimates. Thus, a contradiction arises since $\|\bm\phi_{n}\|_{H}=1$.

\medskip
\textbf{Step 4. (Invertibility)}~The only thing left to prove, at this point, is that $$\bm{\mathcal{L}}_{\bm{\delta},\bm{\xi}}=\left(\mathcal{L}^{1}_{\bm{\delta},\bm{\xi}},\mathcal{L}^{2}_{\bm{\delta},\bm{\xi}}\right):=\left(\Pi^{\perp}_1\mathcal{L}_1,\Pi^{\perp}_2\mathcal{L}_2\right) $$ is an invertible operator. Indeed, since $i_{\lambda}^*:L^{\frac{2N}{N+2}}(\Omega_\lambda)\to H_0^1(\Omega_\lambda)$ is a compact operator, then $$\bm{\mathcal{L}}_{\bm{\delta},\bm{\xi}}=\bm{Id}-\bm{\mathcal K},\quad\text{$\bm{\mathcal K}$ is a compact operator.}$$ From \textbf{Step 1}-\textbf{Step 3}, we have already shown that  $\bm{\mathcal{L}}_{\bm{\delta},\bm{\xi}}$ is injective. Then by Fredholm's alternative theorem, it is also surjective. Therefore, $\bm{\mathcal{L}}_{\bm{\delta},\bm{\xi}}$ is invertible.  Moreover, by \eqref{5-invert-1}, it follows that $\bm{\mathcal{L}}^{-1}_{\bm{\delta},\bm{\xi}}$ is continuous.

\end{proof}

Next,  we can solve problem \eqref{first-equ}.

 \begin{proof}[Proof of Proposition \ref{error-size}]
     From Lemma \ref{invert}, we derive that $\bm{\mathcal{L}}_{\bm{\delta},\bm{\xi}}$ is invertible in $K_{\bm{\delta},\bm{\xi}}^{\perp}$, then \eqref{new-linear} is equivalent to
\begin{equation}
    \label{equiv-phi}
\bm{\phi}=\left(\bm{\mathcal{L}}_{\bm{\delta},\bm{\xi}}\right)^{-1}\Big(\bm{\mathcal{E}}_{\bm{\delta},\bm{\xi}}+\bm{\mathcal{N}}_{\bm{\delta},\bm{\xi}}\left(\bm{\phi}\right)\Big):= \mathcal{T}(\bm{\phi}),
\end{equation}
where the operator $\mathcal{T}:K_{\bm{\delta},\bm{\xi}}^{\perp} \to K_{\bm{\delta},\bm{\xi}}^{\perp}$, $K_{\bm{\delta},\bm{\xi}}^{\perp}:= K_{1}^{\perp}\times K_{2}^{\perp} $. Remark that $\bm{\phi}$ solves equation \eqref{equiv-phi} if and only if $\bm{\phi}$ is a fixed point of the operator $\mathcal{T}$. Thus, we aim to prove that $\mathcal{T}$ is a contraction map.

Firstly, from \eqref{error} and \eqref{perturbed}, we utilize Lemma \ref{embed} and Lemma \ref{ls} to yield that 
\begin{align}
\label{contra}\|\mathcal{T}(\bm{\phi})\|_{H} \leq& c_1 \|\mathcal{E}_{\bm{\delta},\bm{\xi}} + \mathcal{N}_{\bm{\delta},\bm{\xi}}(\bm{\phi})\|_{H}\nonumber\\
\leq& c_2\sum_{i=1}^2 \left( \|\widetilde{\mathcal{E}}^{i}_{\bm{\delta},\bm{\xi}}\|_{{\frac{2N}{N+2}}} +\|\widetilde{\mathcal{E}}^{i}_{\bm{\delta},\bm{\xi}}\|_{{\frac{Ns}{N+2s}}} + \| \widetilde{\mathcal{N}}^{i}_{\bm{\delta},\bm{\xi}}({\phi_1,\phi_2})\|_{{\frac{2N}{N+2}}} + \| \widetilde{\mathcal{N}}^{i}_{\bm{\delta},\bm{\xi}}({\phi_1,\phi_2})\|_{{\frac{Ns}{N+2s}}} \right),
\end{align}
where for $i=1,2$,
\begin{align*}
\widetilde{\mathcal{E}}^{i}_{\bm{\delta},\bm{\xi}}=& \mu_i^{-\frac{N-2}{4}}\left(\left(P_\lambda U_{\delta_i,\xi_i/\lambda}\right)^{p}-U_{\delta_i,\xi_i/\lambda}^{p}\right)+\varepsilon\,\lambda^{2}\mu_i^{-\frac{N-2}{4}}\,P_{\lambda}U_{\delta_i,\xi_i/\lambda}\\&+\beta\lambda^{\frac{6-N}{2}}\mu_i^{-\frac{N-2}{4}}P_{\lambda}U_{\delta_i,\xi_i/\lambda }\sum_{j\neq i}\mu_j^{-\frac{N-2}{4}}\,P_{\lambda}U_{\delta_j,\xi_j/\lambda},
\end{align*}
and 
\begin{align*}
    \widetilde{\mathcal{N}}^{i}_{\bm{\delta},\bm{\xi}}({\phi_1,\phi_2})=& \mu_i\left[\left(\mu_i^{-\frac{N-2}{4}}P_\lambda U_{\delta_i,\xi_i/\lambda}+\phi_i\right)^{p}-\left(\mu_i^{-\frac{N-2}{4}}P_\lambda U_{\delta_i,\xi_i/\lambda}\right)^{p}-p\,\mu_i^{-1}\left(P_\lambda U_{\delta_i,\xi_i/\lambda}\right)^{p-1}\,\phi_i\right]\\&\quad\,\quad+\beta\lambda^{\frac{6-N}{2}}\phi_i\sum_{j\neq i}\phi_j.
\end{align*}
Now we compute each term in the right-hand side of \eqref{contra}. For $\| \widetilde{\mathcal{N}}^{1}_{\bm{\delta},\bm{\xi}}({\phi_1,\phi_2})\|_{{\frac{2N}{N+2}}}$, by Lemma \ref{8} and Lemma \ref{GEQ-7}, we derive that for $N\geq4$,
\begin{align}\label{norm-N}
    \left\| \widetilde{\mathcal{N}}^{1}_{\bm{\delta},\bm{\xi}}({\phi_1,\phi_2})\right\|_{{\frac{2N}{N+2}}}\lesssim&\left\|U^{p-2}_{\delta_1,\xi_1/\lambda}\,\phi^2_1\right\|_{\frac{2N}{N+2}}+{\left\|\phi^{2^*-1}_1\right\|_{\frac{2N}{N+2}}}+|\beta|\lambda^{\frac{6-N}{2}}\left\|\phi_1\,\phi_2\right\|_{\frac{2N}{N+2}}\nonumber\\
    \lesssim&\left\| {\phi_1}\right\|^2_{X},\quad\text{if $|\beta|\lambda^{\frac{6-N}{2}}\to 0$ as $\lambda\to 0$.}
\end{align}
As for $\left\|\widetilde{\mathcal{E}}^{1}_{\bm{\delta},\bm{\xi}}\right\|_{{\frac{2N}{N+2}}}$,
\begin{align*}
\left\|\widetilde{\mathcal{E}}^{1}_{\bm{\delta},\bm{\xi}}\right\|_{{\frac{2N}{N+2}}}\lesssim&\underbrace{\left\|\left(P_\lambda U_{\delta_1,\xi_1/\lambda}\right)^{p}-U_{\delta_1,\xi_1/\lambda}^{p}\right\|_{\frac{2N}{N+2}}}_{:=G_1}+\underbrace{\varepsilon\,\lambda^{2}\left\|U_{\delta_1,\xi_1/\lambda}\right\|_{\frac{2N}{N+2}}}_{:=G_2}\\
&+\underbrace{|\beta|\lambda^{\frac{6-N}{2}}\left\|U_{\delta_1,\xi_1/\lambda}\,U_{\delta_2,\xi_2/\lambda}\right\|_{\frac{2N}{N+2}}}_{:=G_3}.
\end{align*}
Then from Lemma \ref{2}, we derive that 
\begin{align*}
    G_1\lesssim\begin{cases}
    \lambda^{2},\quad &\text{for $N=4$,}\\
        \lambda^{3},\quad &\text{for $N=5$,}\\
        \lambda^{4}\left|\ln{\lambda}\right|,\quad &\text{for $N=6$,}\\
        \lambda^{\frac{N+2}{2}},\quad &\text{for $N\geq 7$.}
    \end{cases}
\end{align*}
From Lemma \ref{esti-bubble}, we derive that
\begin{align*}
    G_2\lesssim\begin{cases}
    \varepsilon\,\lambda,\quad &\text{for $N=4$,}\\
        \varepsilon\,\lambda^{\frac{3}{2}},\quad &\text{for $N=5$,}\\
         \varepsilon\lambda^{2}\left|\ln{\lambda}\right|,\quad &\text{for $N=6$,}\\
        \varepsilon\lambda^{2},\quad &\text{for $N\geq 7$.}
    \end{cases}
\end{align*}
 Set $\lambda x=z$, $\forall x\in \Omega_{\lambda}$. Then,
\begin{align*}
&\left\|U_{\delta_1,\xi_1/\lambda}(x)\,U_{\delta_2,\xi_2/\lambda}(x)\right\|_{\frac{2N}{N+2}}\\=&\lambda^{\frac{N-6}{2}}\left(\displaystyle\int_{B_{\eta}(\xi_1)\bigcup\,B_{\eta}(\xi_2)\bigcup\,\Omega\backslash\left(B_{\eta}(\xi_1)\cup B_{\eta}(\xi_2)\right)}U^{\frac{2N}{N+2}}_{\lambda\delta_1,\xi_1}(z)\,U^{\frac{2N}{N+2}}_{\lambda\delta_2,\xi_2}(z)dz\right)^{\frac{N+2}{2N}}\\
\lesssim&\begin{cases}
\lambda,\quad &\text{for $N=4$,}\\
    \lambda^{\frac{5}{2}},\quad &\text{for $N=5$,}\\
\lambda^{4}\left|\ln{\lambda}\right|,\quad &\text{for $N=6$,}\\
  \lambda^{N-2}   ,\quad &\text{for $N\geq 7$,}\\
\end{cases}
\end{align*}
which implies that
\begin{align*}
    G_3\lesssim\begin{cases}
    |\beta|\lambda^{{2}},\quad &\text{for $N=4$,}\\
        |\beta|\lambda^{{3}},\quad &\text{for $N=5$,}\\
|\beta|\lambda^{4}\left|\ln{\lambda}\right|,\quad &\text{for $N=6$,}\\
       |\beta| \lambda^{\frac{N+2}{2}},\quad &\text{for $N\geq 7$.}
    \end{cases}
\end{align*}
Hence, from the estimates of $G_1$ to $G_3$, we derive that
\begin{align}\label{norm-1}
\left\|\widetilde{\mathcal{E}}^{1}_{\bm{\delta},\bm{\xi}}\right\|_{{\frac{2N}{N+2}}}\lesssim&\begin{cases}\lambda^{2}+\varepsilon\lambda+|\beta|\lambda^{{2}},  \quad &\text{for $N=4$,}\\
  \lambda^{3}+\varepsilon\lambda^{\frac{3}{2}}+|\beta|\lambda^{{3}},  \quad &\text{for $N=5$,}\\
\lambda^{4}\left|\ln{\lambda}\right|+\varepsilon\lambda^{{2}}\left|\ln{\lambda}\right|+|\beta|\lambda^{4}\left|\ln{\lambda}\right|,\quad &\text{for $N=6$,}\\
       \lambda^{\frac{N+2}{2}}+\varepsilon\lambda^{2}  +|\beta| \lambda^{\frac{N+2}{2}},\quad &\text{for $N\geq 7$.}
\end{cases}
\end{align}

In particular, for $N\geq 7$, we need to calculate 
\begin{align*}
\left\|\widetilde{\mathcal{E}}^{1}_{\bm{\delta},\bm{\xi}}\right\|_{{\frac{Ns}{N+2s}}}\lesssim&\underbrace{\left\|\left(P_\lambda U_{\delta_1,\xi_1/\lambda}\right)^{p}-U_{\delta_1,\xi_1/\lambda}^{p}\right\|_{{\frac{Ns}{N+2s}}}}_{:=G_4}+\underbrace{\varepsilon\lambda^{2}\left\|U_{\delta_1,\xi_1/\lambda}\right\|_{\frac{Ns}{N+2s}}}_{:=G_5}\\
&+\underbrace{|\beta|\lambda^{\frac{6-N}{2}}\left\|U_{\delta_1,\xi_1/\lambda}\,U_{\delta_2,\xi_2/\lambda}\right\|_{\frac{Ns}{N+2s}}}_{:=G_6}.
\end{align*}
From Lemma \ref{esti-bubble} and Lemma \ref{2}, we derive that 
\begin{align*}
    G_4\lesssim\lambda^{N-2},\quad\,\quad G_5\lesssim\varepsilon\,\lambda^{2}.
\end{align*}
For $G_{6}$, set $\lambda x=z$, $\forall x\in \Omega_{\lambda}$. Then,
\begin{align*}
&\left\|U_{\delta_1,\xi_1/\lambda}(x)\,U_{\delta_2,\xi_2/\lambda}(x)\right\|_{\frac{Ns}{N+2s}}\\=&\lambda^{N-4-\frac{N}{s}}\left(\displaystyle\int_{B_{\eta}(\xi_1)\bigcup\,B_{\eta}(\xi_2)\bigcup\,\Omega\backslash\left(B_{\eta}(\xi_1)\cup B_{\eta}(\xi_2)\right)}U^{\frac{Ns}{N+2s}}_{\lambda\delta_1,\xi_1}(z)\,U^{\frac{Ns}{N+2s}}_{\lambda\delta_2,\xi_2}(z)dz\right)^{\frac{N+2s}{Ns}}\\
\lesssim&\lambda^{N-2},
\end{align*}
which implies that 
\begin{align*}
    G_{6}\lesssim|\beta|\lambda^{\frac{N+2}{2}}.
\end{align*}
Thus, from the estimates of $G_4$ to $G_6$, we derive that
\begin{align*}
\left\|\widetilde{\mathcal{E}}^{1}_{\bm{\delta},\bm{\xi}}\right\|_{{\frac{Ns}{N+2s}}}\lesssim  \lambda^{N-2}+\varepsilon\,\lambda^{2}+|\beta|\lambda^{\frac{N+2}{2}}.
\end{align*}
Then by \eqref{norm-1} with $N\geq 7$, we conclude that
\begin{align}
    \label{n-7}
\left\|\widetilde{\mathcal{E}}^{1}_{\bm{\delta},\bm{\xi}}\right\|_{{\frac{2N}{N+2}}}+\left\|\widetilde{\mathcal{E}}^{1}_{\bm{\delta},\bm{\xi}}\right\|_{{\frac{Ns}{N+2s}}}\lesssim\lambda^{\frac{N+2}{2}}+\varepsilon\,\lambda^{2}  +|\beta| \lambda^{\frac{N+2}{2}}.
\end{align}

Therefore, from \eqref{norm-1} and \eqref{n-7}, we obtain that if $|\beta|=o\left(\lambda^{\frac{N-6}{2}}\right)\to0$ as $\lambda\to 0$, (no assumptions is needed if $N=4,5$), then
\begin{equation}\label{new-e}
\begin{split}
\left\|\mathcal{E}_{\bm{\delta},\bm{\xi}}\right\|_{{\frac{2N}{N+2}}}+\left\|\mathcal{E}_{\bm{\delta},\bm{\xi}}\right\|_{{\frac{Ns}{N+2s}}}=&\sum_{i=1}^{2}\left(\left\|\widetilde{\mathcal{E}}^{i}_{\bm{\delta},\bm{\xi}}\right\|_{{\frac{2N}{N+2}}}+\left\|\widetilde{\mathcal{E}}^{i}_{\bm{\delta},\bm{\xi}}\right\|_{{\frac{Ns}{N+2s}}}\right)\\\leq&\begin{cases}
C_0\,\varepsilon\,\lambda,  \quad &\text{for $N=4$,}\\
C_0\,\varepsilon\,\lambda^{\frac{3}{2}},  \quad &\text{for $N=5$,}\\
C_0\left(\lambda^{4}\left|\ln{\lambda}\right|+\varepsilon\,\lambda^{2}\left|\ln{\lambda}\right|\right),\quad &\text{for $N=6$,}\\
       C_0\,\lambda^{\frac{N+2}{2}},\quad &\text{for $N\geq 7$.}
\end{cases}
\end{split}
\end{equation}
for some constants $C_0>0$.

\medskip
Now we show that $\mathcal{T}$ maps $\mathcal{M}_N$ into $\mathcal{M}_N$, $N\geq4$. When $N=4$, the set 
\begin{align*}
    &\mathcal{M}_4:=\displaystyle \left\{\bm{\phi}\in K_{\bm{\delta},\bm{\xi}}^{\perp}: \left\|\bm{\phi}\right\|_{H}\leq C\,\varepsilon \lambda\right\},
\end{align*}
for some suitable constants $C>0$. Combining with \eqref{new-e}, the sets $\mathcal{M}_N$ $(N\geq 5)$ can be defined  similarly. Then for any $\bm{\phi} \in \mathcal{M}_N$, from \eqref{contra}, \eqref{norm-N} and \eqref{new-e}, we derive that
\begin{align*}
    \|\mathcal{T}(\bm{\phi})\|_{H} \leq&\begin{cases}
C\,\varepsilon\,\lambda,  \quad &\text{for $N=4$,}\\
C\,\varepsilon\,\lambda^{\frac{3}{2}},  \quad &\text{for $N=5$,}\\
C\left(\lambda^{4}\left|\ln{\lambda}\right|+\varepsilon\,\lambda^{2}\left|\ln{\lambda}\right|\right),\quad &\text{for $N=6$,}\\
       C\,\lambda^{\frac{N+2}{2}},\quad &\text{for $N\geq 7$.}
\end{cases}
\end{align*}
which implies that $\mathcal{T}$ is a mapping from $\mathcal{M}_N$ into $\mathcal{M}_N$, for  $N\geq4$. Moreover, by the definition of $\mathcal{T}$ given in \eqref{equiv-phi}, {we obtain that $\mathcal{T}$ is a contraction mapping}. Indeed, a direct calculation yields that for any $\bm{\phi},\,\bm{\psi}\in \mathcal{M}_N$,
 \begin{align*}
   \left\|\mathcal{T}_{\bm{\delta},\bm{\xi}}(\bm{\phi})-\mathcal{T}_{\bm{\delta},\bm{\xi}}(\bm{\psi})\right\|_{H}\lesssim&\sum_{i=1}^2\left(\left\|\widetilde{\mathcal{N}}^{i}_{\bm{\delta},\bm{\xi}}(\bm{\phi})-\widetilde{\mathcal{N}}^{i}_{\bm{\delta},\bm{\xi}}(\bm{\psi})\right\|_{{\frac{2N}{N+2}}}+\left\|\widetilde{\mathcal{N}}^{i}_{\bm{\delta},\bm{\xi}}(\bm{\phi})-\widetilde{\mathcal{N}}^{i}_{\bm{\delta},\bm{\xi}}(\bm{\psi})\right\|_{{\frac{Ns}{N+2s}}}\right)\\<& \gamma_0 \left\|\bm{\phi}-\bm{\psi}\right\|_{H},
\end{align*}
for some $\gamma_0\in(0,1)$. Hence, the existence result of $\bm{\phi}\in \mathcal{M}_N$, $N\geq 4$, follows by Banach's Fixed Point Theorem. This completes the proof of Proposition \ref{error-size} and \eqref{norm-geq5}.

\end{proof}
\section{Proof of Proposition \ref{supercritical}}\label{BBB}
\renewcommand{\theequation}{B.\arabic{equation}}

\begin{proof}[Proof: the case $N=5$]${}$\\
We only show the case with $i=1$, the case $i=2$ can be derived by similar arguments. Indeed, for $j=0$, a direct computation yields that
\begin{align*}
       \text{RHS of \eqref{first-5}}=&\underbrace{\mu_1\displaystyle{\int_{\Omega_\lambda}}\left(g\left(\mu_1^{-\frac{N-2}{4}}U_{\delta_1,\xi_1/\lambda}\right)-g\left(\mu_1^{-\frac{N-2}{4}}P_{\lambda}U_{\delta_1,\xi_1/\lambda}\right)\right)\,P_{\lambda}\psi^{0}_{\delta_1,\xi_1/\lambda}}_{:=E_1}\\&-\underbrace{\varepsilon\lambda^{2}\mu_1^{-\frac{N-2}{4}}\displaystyle{\int_{\Omega_\lambda}}P_\lambda U_{\delta_1,\xi_1/\lambda}\,P_{\lambda}\psi^{0}_{\delta_1,\xi_1/\lambda}}_{:=E_2}\\&-\underbrace{\beta\lambda^{\frac{6-N}{2}}\mu_1^{{-\frac{N-2}{4}}}\mu_2^{{-\frac{N-2}{4}}}\displaystyle{\int_{\Omega_\lambda}}P_\lambda U_{\delta_1,\xi_1/\lambda}\,P_\lambda U_{\delta_2,\xi_2/\lambda}\,P_{\lambda}\psi^{0}_{\delta_1,\xi_1/\lambda} }_{:=E_3}\\
       &-\underbrace{\mu_{1}\displaystyle{\int_{\Omega_\lambda}}\left(g^{\prime}\left(\mu_1^{-\frac{N-2}{4}}P_{\lambda}U_{\delta_1,\xi_1/\lambda}\right)-g^{\prime}\left(\mu_1^{-\frac{N-2}{4}}U_{\delta_1,\xi_1/\lambda}\right)\right)\phi_{1}\,P_{\lambda}\psi^{0}_{\delta_1,\xi_1/\lambda}}_{:=E_4}\\
       &-\underbrace{\mu_{1}\displaystyle{\int_{\Omega_\lambda}}g^{\prime}\left(\mu_1^{-\frac{N-2}{4}}U_{\delta_1,\xi_1/\lambda}\right)\phi_{1}\left(P_{\lambda}\psi^{0}_{\delta_1,\xi_1/\lambda}-\psi^{0}_{\delta_1,\xi_1/\lambda}\right)}_{:=E_5}\\
       &-\underbrace{\varepsilon\lambda^{2}\displaystyle{\int_{\Omega_\lambda}}\phi_{1}\,P_{\lambda}\psi^{0}_{\delta_1,\xi_1/\lambda}}_{:=E_6}\\&-\underbrace{\beta\lambda^{\frac{6-N}{2}}\mu_2^{{-\frac{N-2}{4}}}\displaystyle{\int_{\Omega_\lambda}}P_\lambda U_{\delta_2,\xi_2/\lambda}\,\phi_1\,P_{\lambda}\psi^{0}_{\delta_1,\xi_1/\lambda} }_{:=E_7}\\&
       -\underbrace{\beta\lambda^{\frac{6-N}{2}}\mu_1^{{-\frac{N-2}{4}}}\displaystyle{\int_{\Omega_\lambda}}P_\lambda U_{\delta_1,\xi_1/\lambda}\,\phi_2\,P_{\lambda}\psi^{0}_{\delta_1,\xi_1/\lambda} }_{:=E_8}\\&-\underbrace{\mu_1\displaystyle{\int_{\Omega_\lambda}}\Bigg[g\left(\mu_1^{-\frac{N-2}{4}}P_{\lambda}U_{\delta_1,\xi_1/\lambda}+\phi_1\right)-g\left(\mu_1^{-\frac{N-2}{4}}P_{\lambda}U_{\delta_1,\xi_1/\lambda}\right)}\\&\quad\,\quad\,\quad\underbrace{-g^{\prime}\left(\mu_1^{-\frac{N-2}{4}}P_{\lambda}U_{\delta_1,\xi_1/\lambda}\right)\phi_1\Bigg]\,P_{\lambda}\psi^{0}_{\delta_1,\xi_1/\lambda}}_{:=E_9}\\&-\underbrace{\beta\lambda^{\frac{6-N}{2}}\displaystyle{\int_{\Omega_\lambda}}\phi_1\,\phi_2\,P_{\lambda}\psi^{0}_{\delta_1,\xi_1/\lambda}}_{:=E_{10}}.
\end{align*}
In the following, we compute each term precisely. From \eqref{main-0} and \eqref{pmain-0}, we have
\begin{align*}
    E_{1}=&\frac{N-2}{2}\,\mu_1^{-\frac{N-2}{4}} A^2 \delta_1^{N-3} H\left(\xi_1, \xi_1\right) \lambda^{N-2}+o\left(\lambda^{N-2}\right).
\end{align*}
From \eqref{pl-main0}, we have
\begin{align*}
    E_{2}=-\mu_1^{-\frac{N-2}{4}}\delta_{1}\,B\,\varepsilon\,\lambda^{2}+o\left(\varepsilon\,\lambda^{2}\right).
\end{align*}

Since $\forall x\in\Omega$, 
\begin{align}
    \label{fact-1}\left|\psi^{0}_{\lambda\delta_i,\xi_i}(x)\right|\lesssim \frac{1}{\lambda}\,U_{\lambda\delta_i,\xi_i}(x),\quad\text{for $i=1,2$},
\end{align}then we have \begin{align}\label{latest-e3}
E_3\lesssim&\left|\beta\right|\lambda\displaystyle\int_{\Omega}U_{\lambda \delta_1,\xi_1}\, U_{\lambda \delta_2,\xi_2}\,|\psi^{0}_{\lambda \delta_1,\xi_1}|
\lesssim\left|\beta\right|\displaystyle\int_{\Omega}U^2_{\lambda\delta_1,\xi_1}\,U_{\lambda \delta_2,\xi_2}\nonumber\\
\lesssim&\left|\beta\right|\left(\displaystyle\int_{B_{\eta}(\xi_1)}U^2_{\lambda \delta_1,\xi_1}\,U_{\lambda \delta_2,\xi_2}+O\left(\lambda^{\frac{3(N-2)}{2}}\right)\right) \lesssim\left|\beta\right|\lambda^{\frac{7}{2}}\left(\displaystyle\int_{\R^N}U^2_{1,0}\right)\nonumber\\=&o\left(\lambda^{3}\right).
\end{align}

\medskip
Now taking into account of $\|\bm{\phi}\|_{H}$ obtained in Lemma \ref{error-size} with $N=5$, we compute the remaining terms. For $E_4$, we have
\begin{align*}
    \left|E_4\right|\lesssim &\left\|\left(g^{\prime}\left(\mu_{1}^{-\frac{N-2}{4}}P_{\lambda} U_{\delta_1, \xi_1 / \lambda}\right)- g^{\prime}\left(\mu_{1}^{-\frac{N-2}{4}} U_{\delta_i, \xi_i / \lambda}\right)\right) P_{\lambda} \psi_i^0\right\|_{\frac{2 N}{N+2}}\|\phi_1\|_{2^*} \\\lesssim & \lambda^{ 3}\,\|\phi_1\|_{X}=o\left(\lambda^{3}\right).
\end{align*}
For $E_5$, from Lemma \ref{7}, we have
\begin{align*}
    \left|E_5\right|\lesssim &\left\|P_{\lambda} \psi_{\delta_1, \xi_1 / \lambda}^0-\psi_{\delta_1, \xi_1 / \lambda}^0\right\|_{\frac{2 N}{N-2}} \|\phi_1\|_{2^*}\,\left\|g^{\prime}\left(\mu_{1}^{-\frac{N-2}{4}} U_{\delta_1, \xi_1 / \lambda}\right)\right\|_{\frac{N}{2}} \\\lesssim & \lambda^{\frac{N-2}{2}}\,\|\phi_1\|_{X}=o\left(\lambda^{3}\right).
\end{align*}
For $E_6$, from Lemma \ref{esimates}, we have
\begin{align*}
    \left|E_6\right|\lesssim \varepsilon\lambda^{2}\left\|P_{\lambda} \psi_{\delta_1, \xi_1 / \lambda}^0\right\|_{2^*} \|\phi_1\|_{\frac{2N}{N+2}}\lesssim\varepsilon\lambda^{2}\,\|\phi_1\|_{X}=o\left(\lambda^{3}\right).
\end{align*}
For $E_7$, by \eqref{fact-1}, a direct computation yields that
\begin{align*}
 \left| E_7\right|\lesssim &|\beta|\lambda\|\phi_1\|_{2^*}\,\left(\displaystyle\int_{\Omega}\left|U_{\lambda\delta_2,\xi_2}\,\psi^{0}_{\lambda\delta_1,\xi_1}\right|^{\frac{2N}{N+2}}\right)^{\frac{N+2}{2N}}  \\
      \lesssim &|\beta|\|\phi_1\|_X\left(\displaystyle\int_{\Omega}\left|U_{\lambda\delta_2,\xi_2}\,U_{\lambda\delta_1,\xi_1}\right|^{\frac{2N}{N+2}}\right)^{\frac{N+2}{2N}}
  \lesssim|\beta|\lambda^{N-2}\|\phi_1\|_X=o\left(\lambda^{3}\right).
\end{align*}
Moreover, from  \eqref{suppose-beta}, we have 
\begin{align}\label{extrabeta-51}
  \left|E_{8}\right|\lesssim&  |\beta|\left\| \phi_{2}\right\|_{X}\left(\displaystyle\int_{\Omega}U^{\frac{4N}{N+2}}_{\lambda\delta_1,\xi_1}\right)^{\frac{N+2}{2N}}\lesssim |\beta|\lambda^{\frac{1}{2}}\left\| \phi_{2}\right\|_{X}\left(\displaystyle\int_{\R^N}U^{\frac{4N}{N+2}}_{1,0}\right)^{\frac{N+2}{2N}}
  =o\left(\lambda^{3}\right).
\end{align}

\medskip 
Now we compute the nonlinear terms involving the remainder term $\bm{\phi}$. For $E_{9}$, since $p={7}/3>2$, then from Lemma \ref{8},    \begin{align*}
     \left|E_{9}\right|\lesssim&\displaystyle{\int_{\Omega_\lambda}}\left|U^{2^*-3}_{\delta_1,\xi_1/\lambda}\phi_1^2+\phi_1^{2^*-1}\right|\,\left|P_{\lambda}\psi^{0}_{\delta_1,\xi_1/\lambda}\right|\lesssim\left\|\phi_{1}\right\|^{2}_{X}=o\left(\lambda^{3}\right).
 \end{align*}
 For $E_{10}$,
 \begin{align*}
     \left|E_{10}\right|\lesssim& \left|\beta\right|\lambda^{\frac{6-N}{2}}\|\phi_1\|_{N/2}\|\phi_2\|_{2^*}\lesssim \left|\beta\right|\lambda^{\frac{6-N}{2}}\|\bm{\phi}\|^2_{H} =o\left(\lambda^{3}\right).
 \end{align*}

 Thus, from the estimates of $E_{1}$-$E_{10}$ and the condition \eqref{suppose-beta}, we obtain that \eqref{7-0} with $N=5$ holds. 

\medskip
 In the following, we compute \eqref{7-n} with $N=5$. Indeed, for $j=1,\ldots,N$, we have
 \begin{align*}
       \text{RHS of \eqref{first-5}}=&\underbrace{\mu_1\displaystyle{\int_{\Omega_\lambda}}\left(g\left(\mu_1^{-\frac{N-2}{4}}U_{\delta_1,\xi_1/\lambda}\right)-g\left(\mu_1^{-\frac{N-2}{4}}P_{\lambda}U_{\delta_1,\xi_1/\lambda}\right)\right)\,P_{\lambda}\psi^{j}_{\delta_1,\xi_1/\lambda}}_{:=Q_1}\\&-\underbrace{\varepsilon\lambda^{2}\mu_1^{-\frac{N-2}{4}}\displaystyle{\int_{\Omega_\lambda}}P_\lambda U_{\delta_1,\xi_1/\lambda}\,P_{\lambda}\psi^{j}_{\delta_1,\xi_1/\lambda}}_{:=Q_2}\\&-\underbrace{\beta\lambda^{\frac{6-N}{2}}\mu_1^{{-\frac{N-2}{4}}}\mu_2^{{-\frac{N-2}{4}}}\displaystyle{\int_{\Omega_\lambda}}P_\lambda U_{\delta_1,\xi_1/\lambda}\,P_\lambda U_{\delta_2,\xi_2/\lambda}\,P_{\lambda}\psi^{j}_{\delta_1,\xi_1/\lambda} }_{:=Q_3}\\
       &-\underbrace{\mu_{1}\displaystyle{\int_{\Omega_\lambda}}\left(g^{\prime}\left(\mu_1^{-\frac{N-2}{4}}P_{\lambda}U_{\delta_1,\xi_1/\lambda}\right)-g^{\prime}\left(\mu_1^{-\frac{N-2}{4}}U_{\delta_1,\xi_1/\lambda}\right)\right)\phi_{1}\,P_{\lambda}\psi^{j}_{\delta_1,\xi_1/\lambda}}_{:=Q_4}\\
       &-\underbrace{\mu_{1}\displaystyle{\int_{\Omega_\lambda}}g^{\prime}\left(\mu_1^{-\frac{N-2}{4}}U_{\delta_1,\xi_1/\lambda}\right)\phi_{1}\left(P_{\lambda}\psi^{j}_{\delta_1,\xi_1/\lambda}-\psi^{j}_{\delta_1,\xi_1/\lambda}\right)}_{:=Q_5}\\
       &-\underbrace{\varepsilon\lambda^{2}\displaystyle{\int_{\Omega_\lambda}}\phi_{1}\,P_{\lambda}\psi^{j}_{\delta_1,\xi_1/\lambda}}_{:=Q_6}\\&-\underbrace{\beta\lambda^{\frac{6-N}{2}}\mu_2^{{-\frac{N-2}{4}}}\displaystyle{\int_{\Omega_\lambda}}P_\lambda U_{\delta_2,\xi_2/\lambda}\,\phi_1\,P_{\lambda}\psi^{j}_{\delta_1,\xi_1/\lambda} }_{:=Q_7}\\&
       -\underbrace{\beta\lambda^{\frac{6-N}{2}}\mu_1^{{-\frac{N-2}{4}}}\displaystyle{\int_{\Omega_\lambda}}P_\lambda U_{\delta_1,\xi_1/\lambda}\,\phi_2\,P_{\lambda}\psi^{j}_{\delta_1,\xi_1/\lambda} }_{:=Q_8}\\&-\underbrace{\mu_1\displaystyle{\int_{\Omega_\lambda}}\Bigg[g\left(\mu_1^{-\frac{N-2}{4}}P_{\lambda}U_{\delta_1,\xi_1/\lambda}+\phi_1\right)-g\left(\mu_1^{-\frac{N-2}{4}}P_{\lambda}U_{\delta_1,\xi_1/\lambda}\right)}\\&\quad\,\quad\,\quad\underbrace{-g^{\prime}\left(\mu_1^{-\frac{N-2}{4}}P_{\lambda}U_{\delta_1,\xi_1/\lambda}\right)\phi_1\Bigg]\,P_{\lambda}\psi^{j}_{\delta_1,\xi_1/\lambda}}_{:=Q_9}\\&-\underbrace{\beta\lambda^{\frac{6-N}{2}}\displaystyle{\int_{\Omega_\lambda}}\phi_1\,\phi_2\,P_{\lambda}\psi^{j}_{\delta_1,\xi_1/\lambda}}_{:=Q_{10}}.
\end{align*}
From \eqref{main-j} and \eqref{pmain-j}, we have
\begin{align*}
    Q_{1}=\mu_1^{-\frac{N-2}{4}}A^2 \delta_1^{N-2} \frac{\partial H}{\partial \xi_1^j}\left(\xi_1,\xi_1\right) \lambda^{N-1}+o\left(\lambda^{N-1}\right).
\end{align*}
By \eqref{pl-mainj}, we get that
\begin{align*}
    Q_{2}=o\left(\varepsilon\lambda^{3}\right).
\end{align*}
For $Q_3$, since \begin{align}
    \label{fact-2}\left|\psi^{j}_{\lambda\delta_i,\xi_i}(x)\right|\lesssim \frac{1}{\lambda}\,U_{\lambda\delta_i,\xi_i}(x),\quad\text{for $j=1,\ldots,N$ and $i=1,2$},
\end{align}
then we use similar arguments as $E_3$ to derive that if $\beta$ verifies \eqref{suppose-beta}, then
\begin{align}\label{latest-ee3}
    Q_3
    \lesssim&\left|\beta\right|\lambda^{\frac{7}{2}}\left(\displaystyle\int_{\R^N}U^2_{1,0}\right)=o\left(\lambda^{4}\right).
\end{align}

\medskip
Now by similar arguments as the estimates of $E_4$-$E_{10}$, we compute the remaining terms. For $Q_4$, we have
    \begin{align*}
    \left|Q_4\right|\lesssim  \lambda^{ 3}\,\|\phi_1\|_{X}=o\left(\lambda^{4}\right).
\end{align*}
For $Q_5$, from Lemma \ref{7}, we have
\begin{align*}
    \left|Q_5\right|\lesssim &\left\|P_{\lambda} \psi_{\delta_1, \xi_1 / \lambda}^j-\psi_{\delta_1, \xi_1 / \lambda}^j\right\|_{\frac{2 N}{N-2}} \|\phi_1\|_{2^*}\,\left\|g^{\prime}\left(\mu_{1}^{-\frac{N-2}{4}} U_{\delta_1, \xi_1 / \lambda}\right)\right\|_{\frac{N}{2}} \lesssim  \lambda^{ \frac{N}{2}}\,\|\phi_1\|_{X}=o\left(\lambda^{4}\right).
\end{align*}
For $Q_6$, we have
\begin{align*}
    \left|Q_6\right|\lesssim \varepsilon\lambda^{2}\,\|\phi_1\|_{X}=o\left(\lambda^{4}\right).
\end{align*}
For $Q_7$, from \eqref{fact-2}, we have
\begin{align*}
 \left| Q_7\right|\lesssim|\beta|\lambda^{N-2}\|\phi_1\|_X=o\left(\lambda^{4}\right).
\end{align*}
In particular, for $Q_{8}$, by \eqref{suppose-beta}, we derive that
\begin{align*}
  \left|Q_{8}\right|\lesssim&|\beta|\lambda^{\frac{6-N}{2}}\left\| \phi_{2}\right\|_{X}\left(\displaystyle\int_{\R^N}U^{\frac{4N}{N+2}}_{1,0}\right)^{\frac{N+2}{2N}}=o\left(\lambda^{4}\right).
\end{align*}

For the nonlinear terms involving the remainder term, It is easy to get that    \begin{align*}
     \left|Q_{9}\right|=o\left(\lambda^{4}\right), \quad\left|Q_{10}\right|=o\left(\lambda^{4}\right).
 \end{align*}
 
 Hence, from the estimates of $Q_{1}$-$Q_{10}$ and the condition \eqref{suppose-beta}, we obtain that \eqref{7-n} with $N=5$ holds. 
\end{proof}
 
\vspace{0.2cm}\begin{Rem}\label{similarities}
    For simplicity, we point out some similarities between $N=5$ and $N>5$, then we omit some estimates subsequently. Firstly, for the leading term with $j=0$, the estimates of $E_1$ and $E_2$, with $N=6$ or $N\geq 7$, can be derived  similarly by Lemma \ref{lma-a-9} and Lemma \ref{lma-a-10}. Especially,  the estimate of $E_3$ is the higher-order term  since the smallness of $|\beta|$, given in \eqref{suppose-beta}, is required in each dimension. Analogously, for the leading term with $j=1,\ldots,N$, we derive the similar estimates of $Q_1$ and $Q_2$ when $N=6$ or $N\geq 7$. Due to the assumption on $|\beta|$, $Q_3$ is the higher-order term.

    However, since the Lotka-Volterra interaction term plays distinct roles depending on the dimension of space, we need to calculate the terms $E_7$, $E_8$, $Q_7$ and $Q_8$ when $N=6$ and $N\geq 7$. This will determine the precise order of $|\beta|$, which is closely related to $\lambda:=\lambda(\varepsilon)$.  

For the nonlinear part of \eqref{first-5} involving the remainder term $\bm{\phi}$, we directly use Lemma \ref{8}, Lemma \ref{GEQ-7} and \eqref{norm-geq5} to compute $E_9$ and $E_{10}$ with $N=6$ and $N\geq 7$, which will be the higher-order term. 
\end{Rem}

\vspace{0.2cm}

\begin{proof}[Proof: the case $N=6$]${}$\\ The desired results can be obtained along the lines of the case $N=5$. For simplicity, we only show the differences between $N=6$ and $N=5$, which have been stated in Remark \ref{similarities}. Firstly, we consider the linear terms involving the remainder term. For $E_7$ with $N=6$, by the condition on $\beta$ \eqref{suppose-beta} and the estimates of $\|\bm{\phi}\|_{H}$ given in \eqref{norm-geq5}, a direct computation yields that
\begin{align}\label{6-e7}
 \left| E_7\right|
      \lesssim &|\beta|\|\phi_1\|_X\left(\displaystyle\int_{\Omega}\left|U_{\lambda\delta_2,\xi_2}\,U_{\lambda\delta_1,\xi_1}\right|^{\frac{2N}{N+2}}\right)^{\frac{N+2}{2N}}
  \nonumber\\\lesssim&|\beta|\lambda^{\frac{N+2}{2}}\left|\ln{\lambda}\right|\,\|\phi_1\|_X=o(\lambda^{4}).
\end{align}
For $E_8$, from \eqref{suppose-beta}, we have
   \begin{align*}
 \left| E_8\right|\lesssim &|\beta|\left\| \phi_{2}\right\|_{X}\left(\displaystyle\int_{\R^N}U^{\frac{4N}{N+2}}_{1,0}\right)^{\frac{N+2}{2N}}=o(\lambda^{4}).
\end{align*}

For the nonlinear terms involving the remainder term,  by Lemma \ref{8} with $p=2$, we have    \begin{align*}
     \left|E_{9}\right|\lesssim&\displaystyle{\int_{\Omega_\lambda}}\left|\phi_1^{2}\right|\,\left|\psi^{0}_{\delta_1,\xi_1/\lambda}\right|\lesssim\left\|\phi_{1}\right\|^{2}_{X}=o(\lambda^{4}).
 \end{align*}
 For $E_{10}$, we have
 \begin{align*}
     \left|E_{10}\right|\lesssim& \left|\beta\right|\|\phi_1\|_{2^*}\|\phi_2\|_{2^*}\lesssim  \left|\beta\right|\|\bm{\phi}\|^2_{H}=o(\lambda^{4}).
 \end{align*}

  Hence, from the estimates of $E_{1}$-$E_{10}$ and the condition on $|\beta|$ given in \eqref{suppose-beta}, we obtain that \eqref{7-0} with $N=6$ holds. 

In the following, we state the differences between $N=6$ and $N=5$ while calculating $Q_{1}$-$Q_{10}$.
For $Q_7$, we use similar arguments as $E_7$, obtained in \eqref{6-e7}, to yield that
\begin{align*}
 \left| Q_7\right|\lesssim|\beta|\lambda^{\frac{N+2}{2}}\left|\ln{\lambda}\right|\,\|\phi_1\|_X=o(\lambda^{5}).
\end{align*}
For $Q_8$, from \eqref{suppose-beta}, we have
   \begin{align*}
 \left| Q_8\right|\lesssim&|\beta|\left\| \phi_{2}\right\|_{X}\left(\displaystyle\int_{\R^N}U^{\frac{4N}{N+2}}_{1,0}\right)^{\frac{N+2}{2N}}=o(\lambda^{5}).
\end{align*}

Thus  by the estimates of $Q_{1}$-$Q_{10}$ and the condition \eqref{suppose-beta} on $|\beta|$, we obtain that \eqref{7-n} with $N=6$ holds. 

\end{proof}
 
\begin{proof}[Proof: the case $N\ge7$]${}$\\
 For simplicity, we only show the differences between $N=7$ and $N=5$ while solving the reduced problem. For $E_7$, a direct computation yields that
\begin{align}
 \left| E_7\right|
      \lesssim |\beta|\|\phi_1\|_X\left(\displaystyle\int_{\Omega}\left|U_{\lambda\delta_2,\xi_2}\,U_{\lambda\delta_1,\xi_1}\right|^{\frac{2N}{N+2}}\right)^{\frac{N+2}{2N}}
\lesssim|\beta|\lambda^{\frac{N+2}{2}}\,\|\phi_1\|_X=o(\lambda^{N-2}).\label{7-e7}
\end{align}
For $E_8$, by \eqref{extrabeta-51} and the assumption on $|\beta|$ \eqref{suppose-beta}, we have
   \begin{align}\label{7-e8}
 \left| E_8\right|\lesssim &|\beta|\lambda^{\frac{6-N}{2}}\left\| \phi_{2}\right\|_{X}\left(\displaystyle\int_{\R^N}U^{\frac{4N}{N+2}}_{1,0}\right)^{\frac{N+2}{2N}}=o(\lambda^{N-2}).
\end{align}

   For the nonlinear terms involving the remainder term, by Lemma \ref{GEQ-7}, {\begin{align*}
\left|E_{9}\right|\lesssim\|{\phi}_1\|^2_{X}=o(\lambda^{N-2}).
 \end{align*}}
 For $E_{10}$, since $\frac{4N}{N-2}\in (2^*,s)$, by interpolation, we have
 \begin{align*}
     \left|E_{10}\right|\lesssim& \left|\beta\right|\lambda^{\frac{6-N}{2}}\|\phi_1\|_{\frac{4N}{N-2}}\|\phi_2\|_{\frac{4N}{N-2}}\left\| \psi_{\delta_1, \xi_1 / \lambda^\alpha}^0\right\|_{\frac{2N}{N-2}}\\
      \lesssim& \left|\beta\right|\lambda^{\frac{6-N}{2}}\|\bm{\phi}\|^2_{X} =o(\lambda^{N-2}).
 \end{align*}

  Hence, by the estimates of $E_{1}$-$E_{10}$ and the condition on $|\beta|$, we obtain that \eqref{7-0} with $N\geq 7$ holds. 

In the following, we state the differences between $N=7$ and $N=5$ while calculating $Q_{1}$-$Q_{10}$. For $Q_7$, we use similar arguments as $E_7$, obtained in \eqref{7-e7}, to yield that
\begin{align*}
 \left| Q_7\right|\lesssim|\beta|\lambda^{\frac{N+2}{2}}\,\|\phi_1\|_X=o(\lambda^{N-1}).
\end{align*}

 For $Q_8$, by \eqref{7-e8} and the assumptions on $\beta$ \eqref{suppose-beta}, we have
   \begin{align*}
 \left| Q_8\right|\lesssim&|\beta|\lambda^{\frac{6-N}{2}}\left\| \phi_{2}\right\|_{X}\left(\displaystyle\int_{\R^N}U^{\frac{4N}{N+2}}_{1,0}\right)^{\frac{N+2}{2N}}=o(\lambda^{N-1}).
\end{align*}

Thus by the estimates of $Q_{1}$-$Q_{10}$, we conclude that \eqref{7-n} holds if  $|\beta|$ satisfies \eqref{suppose-beta} with $N\geq 7$. 

\end{proof}
\section{Some useful Estimates }\label{CCC}
\renewcommand{\theequation}{C.\arabic{equation}}

For any $\xi \in \mathbb{R}^N$ and $\delta>0$, set
$$
P\, U_{\lambda\delta, \xi}(x)=i_{\Omega}^*\left(U_{\lambda\delta, \xi}^p\right)(x), \quad x \in \Omega,\quad p=2^*-1.
$$
and
$$
P_{\lambda} U_{ \delta, \xi/\lambda}(y)=i_{{\lambda}}^*\left(U_{\delta, \xi/\lambda}^p\right)(y), \quad y \in \Omega_{\lambda}:=\Omega/\lambda,
$$
where $P:H^{1}(\R^N)\to H^{1}_0(\Omega)$ and $P_\lambda:H^{1}(\R^N)\to H^{1}_0(\Omega_\lambda)$ are the projection operators. In particular, it holds that
$$
P U_{\lambda\delta, \xi}(x)=\lambda^{- \frac{N-2}{2}} P_{\lambda} U_{\lambda, \xi/\lambda}\left(\frac{x}{\lambda}\right), \quad x \in \Omega,
$$
and
$$P \psi^{j}_{\lambda \delta, \xi}(x)=\lambda^{-\frac{N}{2}} P_{\lambda} \psi^{j}_{\lambda, \xi/\lambda}\left(\frac{x}{\lambda}\right), \quad x \in \Omega.$$

\medskip
Set $\xi_i=\lambda y_i$, for $i=1,2$. We derive the following estimates.
\begin{Lem}\cite[Lemma 5.1]{ade}\cite[Lemma 6.1]{musso}\label{1} Let $\xi_i \in \Omega$, $i=1,2$. It holds that $\forall x \in \Omega$,
\begin{align}
   P U_{\lambda\delta_i, \xi_i}(x)&=U_{\lambda\delta_i, \xi_i}(x)-A\left(\lambda\delta_i\right)^{\frac{N-2}{2}} H(x, \xi_i)+o\left(\lambda^{\frac{N-2}{2}}\right),\label{varphi-1}\\P \psi_{\lambda \delta_i, \xi_i}^0(x)&=\psi_{\lambda \delta_i, \xi_i}^0(x)-A\left(\lambda\delta_i\right)^{\frac{N-2}{2}} H(x, \xi_i)+o\left(\lambda^{\frac{N-2}{2}}\right),\label{varphi-2}\\
   P \psi_{\lambda \delta_i, \xi_i}^j(x)&=\psi_{\lambda \delta_i, \xi_i}^j(x)-A\left(\lambda\delta_i\right)^{\frac{N}{2}} \frac{\partial H}{\partial \xi_j}(x, \xi_i) +o\left(\lambda^{\frac{N}{2}} \right), \quad j=1, \ldots, N,\nonumber
\end{align}
where and $A$ is defined by\begin{align}
    \label{a}A:=\displaystyle\int_{\mathbb{R}^N} U_{1,0}^p(z) dz.
\end{align} 
Moreover, 
\begin{align*}
   P U_{\lambda\delta_i, \xi_i}(x)&=A\left( \lambda\delta_i\right)^{\frac{N-2}{2}} G(x, \xi_i)+o\left(\lambda^{\frac{N-2}{2}}\right), \\P \psi_{\lambda \delta_i, \xi_i}^0(x)&= \frac{N-2}{2}A\left(\lambda \delta_i\right)^{\frac{N}{2}-2} G(x, \xi_i)+o\left(\lambda^{\frac{N}{2}-2}\right), \quad x \in \Omega,\\  P \psi_{\lambda \delta_i, \xi_i}^j(x)&=A\left(\lambda\delta_i\right)^{\frac{N-2}{2}} \frac{\partial G}{\partial \xi_i^j}(x, \xi_i)+o\left(\lambda^{ \frac{N-2}{2}}\right), \quad x \in \Omega,
\end{align*}
as $\lambda \to 0$ uniformly on compact sets of $\Omega \backslash\{\xi_i\}$, where $A$ is given by \eqref{a}.
\end{Lem}
\medskip
Now, recall the definitions of \eqref{limit-def} and \eqref{ker-ele}. Then the following estimates can be derived.
\begin{Lem}\cite[Remark 5.2]{musso}\label{esimates}
    There exists $c>0$ such that for any $\lambda>0$, $i=1,2$ and $j=0,1, \ldots, N$, it holds that

$$
\left\|P_{\lambda} U_i\right\|_{H^{1}_0(\Omega_\lambda)} \leq c, \quad\left\|P_{\lambda} U_i\right\|_{\frac{2 N}{N-2}} \leq c \quad \text { and } \quad\left\|P_{\lambda} \psi_i^j\right\|_{\frac{2 N}{N-2}} \leq c.
$$
Moreover,
$$
\begin{gathered}
\left\|P_{\lambda} \psi_i^j\right\|_{\frac{2 N}{N+2}} \leq c, \quad \text { if } j \neq 0, \\
\left\|P_{\lambda} \psi_i^0\right\|_{\frac{2 N}{N+2}} \leq \begin{cases}c, & \text { if } N \geq 7, \\
c \lambda^{-\frac{1}{2}}, & \text { if } N=5,6.\end{cases}
\end{gathered}
$$
\end{Lem}
\medskip
\begin{Lem}\cite[Lemma 5.3]{ade}\label{esti-bubble}
 Let $N \geq 3$. For any $\eta \in(0,1)$ and for any $\lambda_0>0$, there exists $c>0$ such that for any $(\delta_i, \xi_i) \in \mathcal{O}_\delta$ and for any $\lambda \in\left(0, \lambda_0\right)$, it holds that

$$
\chi_2(\lambda):=\left\|\left(P_{\lambda} U_{\delta_i, \xi_i / \lambda}\right)^q\right\|_{\frac{2 N}{N+2}} \leq \begin{cases}c & \text { if } q>\frac{p}{2}, \\ c|\ln \lambda| & \text { if } q=\frac{p}{2}, \\ c \lambda^{- \frac{N+2}{2}\left(1-\frac{2 q}{p}\right)} & \text { if } 1 \leq q<\frac{p}{2} .\end{cases}
$$
Moreover, if $q>\frac{2}{N-2}$, then
$$
\begin{aligned}
\left\|\left(P_{\lambda} U_{\delta_i, \xi_i / \lambda}\right)^q\right\|_{\frac{N s}{N+2 s}} \leq c.
\end{aligned}$$
\end{Lem}

\begin{Lem}\cite[Lemma 5.2]{ade}\label{2} For any $\eta \in(0,1)$ and for any $\lambda_0>0$, there exists $c>0$ such that for any $(\delta_i, \xi_i) \in \mathcal{O}_\delta$ given by \eqref{ker-ele}, and for any $\lambda \in\left(0, \lambda_0\right)$, it holds that

$$
\chi_1(\lambda):=\left\|g\left(P_{\lambda} U_{\delta_i, \xi_i / \lambda}\right)-g\left(U_{\delta_i, \xi_i / \lambda}\right)\right\|_{\frac{2 N}{N+2}} \leq \begin{cases}c \lambda^{ \frac{N+2}{2}} & \text { if } N \geq 7 \\ c \lambda^{4 }|\ln{\lambda} | & \text { if } N=6 \\ c \lambda^{N-2} & \text { if } N=3,4,5\end{cases}
$$
where $g(\omega)=\left(\omega^{+}\right)^p$, for any $\omega\in X$. Moreover,
$$
\left\|g\left(P_{\lambda} U_{\delta, \xi / \lambda}\right)-g\left(U_{\delta, \xi / \lambda}\right)\right\|_{\frac{N s}{N+2 s}} \leq c \lambda^{N-2}.
$$
\end{Lem}

%\begin{Lem}\cite[Lemma 5.4]{musso}\label{3}
% For any $\eta \in(0,1)$ and for any $\lambda_0>0$, there exists $c>0$ such that for any $(\delta_i, \xi_i) \in \mathcal{O}_\delta$ and for any $\lambda \in\left(0, \lambda_0\right)$, it holds that
%$$
%\left\|\left(g^{\prime}\left(P_{\lambda} U_{\delta_i, \xi_i / \lambda}\right)- g^{\prime}\left( U_{\delta_i, \xi_i / \lambda}\right)\right) P_{\lambda} \psi_i^j\right\|_{\frac{2 N}{N+2}} \leq C \lambda^{ \frac{N+2}{2}},\quad j=0,1, \ldots, N,
%$$

%\end{Lem}

\begin{Lem}\cite[Lemma 5.4]{ade}
Let $N \geq 3$. For any $\eta \in(0,1)$ and for any $\lambda_0>0$, there exists $c>0$ such that for any $(\delta_i, \xi_i) \in \mathcal{O}_\delta$ and for any $\lambda \in\left(0, \lambda_0\right)$, it holds that $j=1, \ldots, N$

$$
\int_{\Omega_{\lambda}}\left|\left(P_{\lambda} U_{\delta_i, \xi_i / \lambda}\right)^q-\left(U_{\delta_i, \xi_i/ \lambda}\right)^q\right|\left|\psi_{\delta_i, \xi_i /\lambda}^j\right| \leq \begin{cases}c \lambda^{[(N-2) q-1]} & \text { if } q<\frac{N-1}{N-2} \\ c \lambda^{N-2}|\ln \lambda| & \text { if } q=\frac{N-1}{N-2} \\ c \lambda^{N-2} & \text { if } q>\frac{N-1}{N-2}\end{cases}
$$
and for $j=0$,
$$
\int_{\Omega_{\lambda}}\left|\left(P_{\lambda} U_{\delta_i, \xi_i/ \lambda}\right)^q-\left(U_{\delta_i, \xi_i/ \lambda}\right)^q\right|\left|\psi_{\delta_i, \xi_i/ \lambda}^0\right| \leq \begin{cases}c \lambda^{[(N-2) q-2]} & \text { if } q<\frac{N}{N-2} \\ c \lambda^{N-2}|\ln{\lambda}| & \text { if } q=\frac{N}{N-2} \\ c \lambda^{N-2} & \text { if } q>\frac{N}{N-2}\end{cases}
$$

\end{Lem}

\begin{Lem}\cite[Lemma 6.4]{musso}\label{7}
     Let $N \geq 3$. For any $\eta \in(0,1)$ and for any $\lambda_0>0$, there exists $c>0$ such that for any $(\delta_i, \xi_i) \in \mathcal{O}_\delta$ and for any $\lambda \in\left(0, \lambda_0\right)$, it holds that 
     $$
\begin{gathered}
\left\|P_{\lambda} \psi_{\delta_i, \xi_i / \lambda}^j-\psi_{\delta_i, \xi_i / \lambda}^j\right\|_{\frac{2 N}{N-2}} \leq c \lambda^{\frac{N}{2}}, \quad \text { for } j=1, \ldots, N, \\
\left\|P_{\lambda} \psi_{\delta_i, \xi_i / \lambda}^0-\psi_{\delta_i, \xi_i / \lambda}^0\right\|_{\frac{2 N}{N-2}} \leq c \lambda^{\frac{N-2}{2}}.
\end{gathered}
$$
\end{Lem}

\begin{Lem}\cite[Lemma 6.6]{musso}
If $j=0$ and $i=1,2$, then
\begin{equation}
    \label{main-0}\displaystyle\int_{\Omega_\lambda} U_{\delta_i,\xi_i/\lambda}^p\, P_{\lambda} \psi_{\delta_i,\xi_i/\lambda}^0=-\frac{N-2}{2} A^2 \delta_i^{N-3} H\left(\xi_i, \xi_i\right) \lambda^{N-2}+o\left(\lambda^{N-2}\right),
\end{equation}
as $\lambda \rightarrow 0$ uniformly with respect to $\left(\bm{\delta},\bm{\xi} \right)\in \mathcal{O}_\delta$. Moreover, if $j=1, \ldots, N$ and $i=1, 2$, then
\begin{equation}
    \label{main-j}\displaystyle\int_{\Omega_\lambda} U_{\delta_i,\xi_i/\lambda}^p\, P_{\lambda} \psi_{\delta_i,\xi_i/\lambda}^j=-A^2 \delta_i^{N-2} \frac{\partial H}{\partial \xi_i^j}\left(\xi_i,\xi_i\right) \lambda^{N-1}+o\left(\lambda^{N-1}\right),
\end{equation}
as $\lambda \rightarrow 0$ uniformly with respect to $\left(\bm{\delta},\bm{\xi} \right) \in \mathcal{O}_\delta$. 
\end{Lem} 

\begin{Lem}\cite[Lemma 6.7]{musso}\label{lma-a-9}
If $j=0$ and $i=1,2$, then
\begin{equation}
    \label{pmain-0}\displaystyle\int_{\Omega_\lambda} \left(P_{\lambda}U_{\delta_i,\xi_i/\lambda}\right)^p\, P_{\lambda} \psi_{\delta_i,\xi_i/\lambda}^0=-{(N-2)} A^2 \delta_i^{N-3} H\left(\xi_i, \xi_i\right) \lambda^{N-2}+o\left(\lambda^{N-2}\right),
\end{equation}
as $\lambda \to 0$ uniformly with respect to $\left(\bm{\delta},\bm{\xi} \right)\in \mathcal{O}_\delta$. Moreover, if $j=1, \ldots, N$ and $i=1, 2$, then
\begin{equation}
    \label{pmain-j}\displaystyle\int_{\Omega_\lambda} \left(P_{\lambda} U_{\delta_i,\xi_i/\lambda}\right)^p\, P_{\lambda} \psi_{\delta_i,\xi_i/\lambda}^j=-2A^2 \delta_i^{N-2} \frac{\partial H}{\partial \xi_i^j}\left(\xi_i,\xi_i\right) \lambda^{N-1}+o\left(\lambda^{N-1}\right),
\end{equation}
as $\lambda \to 0$ uniformly with respect to $\left(\bm{\delta},\bm{\xi} \right) \in \mathcal{O}_\delta$. 
\end{Lem} 
\begin{Lem}\cite[Lemma 6.8]{musso}\label{lma-a-10}
  If $j=0$ and $i=1,2$, then
\begin{equation}
    \label{pl-main0}\displaystyle\int_{\Omega_\lambda} P_{\lambda} U_{\delta_i,\xi_i/\lambda} P_{\lambda} \psi_{\delta_i,\xi_i/\lambda}^0=\delta_i B+o(1),
\end{equation}  as $\lambda \to 0$ uniformly with respect to $\left(\bm{\delta},\bm{\xi} \right) \in \mathcal{O}_\delta$. Moreover, if $j=1, \ldots, N$ and $i=1, 2$, then
\begin{equation}
    \label{pl-mainj}\displaystyle\int_{\Omega_\lambda}P_{\lambda} U_{\delta_i,\xi_i/\lambda} P_{\lambda} \psi_{\delta_i,\xi_i/\lambda}^j=o\left(\lambda\right),
\end{equation}
as $\lambda \to 0$ uniformly with respect to $\left(\bm{\delta},\bm{\xi} \right) \in \mathcal{O}_\delta$.
\end{Lem}

\begin{Lem}\label{8} For any $a>0$ and $b>0$, one has

$$
\begin{gathered}
(a+b)^p=a^p+p a^{p-1} b+O\left(b^p\right), \quad \text { if } p \in(1,2], \\
(a+b)^p=a^p+p a^{p-1} b+\frac{p(p-1)}{2} a^{p-2} b^2+O\left(b^p\right), \quad\text { if } p>2,
\end{gathered}
$$

and

$$
(a+b)^p=a^p+O\left(b^p\right), \quad \text { if } p \in(0,1] .
$$

\end{Lem}

\vspace{0.2cm}
\begin{Lem}
    \label{GEQ-7}Since $2^*-1<2$ when $N >6$,
$$
\frac{(1+x)^{2^*-1}-1-\left(2^*-1\right) x}{x^2} \rightarrow C \geq 0, \quad x \rightarrow 0,
$$
and it holds that
$$
\frac{(1+x)^{2^*-1}-1-\left(2^*-1\right) x}{x^2} \rightarrow C \geq 0, \quad x \rightarrow \infty,
$$
$$
\Rightarrow\left|\frac{(1+x)^{2^*-1}-1-\left(2^*-1\right) x}{x^2}\right| \leq C, \quad \forall x \in \mathbb{R}.
$$
\end{Lem}

\vspace{0.2cm}

\bibliographystyle{abbrv}
\bibliography{llv-new}
\end{document}